\documentclass[aap]{imsart}

\RequirePackage{amsthm,amsmath,amsfonts,amssymb}
\RequirePackage[numbers]{natbib}

\startlocaldefs
\usepackage{tikz}
\usepackage{graphicx}     
\usepackage{hyperref}  
\usepackage{mathtools}
\mathtoolsset{showonlyrefs}

\theoremstyle{definition}

\theoremstyle{plain}
\newtheorem{theorem}{Theorem}
\newtheorem{proposition}{Proposition}
\newtheorem{lemma}{Lemma}
\theoremstyle{definition}

\newtheorem{definition}{Definition}
\newtheorem{remark}{Remark}
\theoremstyle{plain}
\newtheorem{corollary}{Corollary}

\DeclareMathOperator*{\argmin}{arg\,min}
\DeclareMathOperator*{\var}{var}
\DeclareMathOperator*{\cov}{cov}
\DeclareMathOperator*{\tr}{tr}

\endlocaldefs

\begin{document}

\begin{frontmatter}
\title{Simple and Sharp Generalization Bounds via Lifting}
\runtitle{Generalization Bounds via Lifting}
\runauthor{J. Liu}
\pdfsubject{AAP2441: accepted manuscript for the Annals of Applied Probability}

\begin{aug}
\author[A]{\fnms{Jingbo}~\snm{Liu}\ead[label=e1]{jingbol@illinois.edu}}
\address[A]{Department of Statistics, University of Illinois
\printead[presep={,\ }]{e1}}
\end{aug}

\begin{abstract}
We bound the supremum of stochastic processes using a rate-distortion functional analogous to classical entropy integrals. 
A lifting argument embeds the process into a higher-dimensional space with permutation symmetry, reducing the analysis to a stationary process. 
This gives a two-sided characterization of the supremum of Gaussian processes, including cases where Dudley's entropy integral is not sharp.
Combining the sharpness of Dudley's integral for stationary processes with information-theoretic arguments gives a proof of the majorizing measure theorem. 
The method also yields localized generalization bounds without a slicing argument. 
As an application, we derive risk bounds for empirical risk minimization over Lipschitz images of Sobolev ellipsoids. Our proof does not use the majorizing measure theorem.
\end{abstract}

\begin{keyword}[class=MSC]
\kwdgroup[type=primary]{\kwd{60G15}\kwd{60E15}}
\kwdgroup[type=secondary]{\kwd{62G08}\kwd{62C20}\kwd{62F12}\kwd{94A17}}
\end{keyword}

\begin{keyword}
\kwd{suprema of stochastic processes}
\kwd{majorizing measures}
\kwd{rate-distortion theory}
\kwd{information-theoretic methods}
\kwd{empirical processes}
\kwd{Sobolev ellipsoids}
\end{keyword}

\end{frontmatter}

\section{Introduction}

Bounding the supremum of stochastic processes is a fundamental problem in probability theory, with important applications in statistics and high-dimensional inference. 
Classical tools such as Dudley's entropy integral provide general upper bounds in terms of covering numbers, and have been widely used in settings ranging from nonparametric regression \citep{van1990estimating, wainwright2019high} to sparse high-dimensional models \citep{raskutti2011minimax}. 
However, Dudley’s bound is not sharp for general Gaussian processes, and its localized variants typically rely on slicing arguments that require bounds on local covering numbers.

Sharpness is regained by Talagrand’s majorizing measure theorem \citep{talagrand1987regularity}, which yields two-sided bounds on the supremum of Gaussian processes via the $\gamma_2$ functional. 
The majorizing measure theorem has nontrivial consequences, such as sub-Gaussian comparison \cite{van2025subgaussian}.
Nevertheless, the $\gamma_2$
formulation is often difficult to evaluate in concrete problems such as statistical generalization bounds, and its proofs use chaining constructions such as admissible nets.

In this paper, we represent the supremum of stochastic processes through a rate-distortion functional of a probability measure on the index set.
This yields a new proof of the majorizing measure theorem, a comparison principle for each prescribed index distribution, and a direct method for localization.
Our \emph{lifting} construction embeds the process into a higher-dimensional space with permutation symmetry, producing a stationary process for which entropy integrals are sharp.

For a centered Gaussian process $(X_t)_{t\in T}$ on a finite index set $T$ and a measure $\mu$ on $T$, this approach yields
\begin{align}
\sup\mathbb{E}[X_Z]
\asymp\int_0^{\infty}
\sqrt{\inf_{P_{Z\hat{Z}}\in\Pi_{\sigma}(\mu)} I(Z;\hat{Z})}\,d\sigma,
\label{e1}
\end{align}
up to universal constants. 
Here the supremum is taken over all random variables $Z$ (possibly dependent on the process) with distribution $P_Z=\mu$,
$\Pi_{\sigma}(\mu)$ denotes the set of joint distributions $P_{Z\hat{Z}}$ with $P_Z=P_{\hat{Z}}=\mu$ and satisfying $\mathbb{E}[d^2(Z,\hat{Z})]\le \sigma^2$, 
and $I(Z;\hat{Z})$ denotes the mutual information.
We will see that \eqref{e1} continues to hold when the constraint $P_{\hat{Z}}=\mu$ is dropped, in which case the infimum in \eqref{e1} becomes the rate-distortion function (Definition~\ref{def_rd}).
While connections between rate-distortion and covering numbers are classical \citep{thomas2006elements,CsiszarKorner1981}, we are not aware of prior work establishing \eqref{e1} for stochastic processes.

Fernique proved Dudley's integral sharp for stationary Gaussian processes, using a reverse-chaining argument that mirrors the proof of the upper bound \citep{fernique1975regularite}.
Applied to the stationary process obtained by lifting (Definition~\ref{def_stat}), this gives the sharp characterization \eqref{e1}.
Optimizing over $\mu$, applying a minimax theorem for an entropic functional, and using the data processing inequality then proves the lower bound in the majorizing measure theorem through Fernique's earlier theorem \cite{talagrand1987regularity,maurey}.
Figure~\ref{fig1} illustrates these implications.
The reduction to stationary processes may also help bound other stochastic processes using reverse chaining instead of admissible nets \cite{van2014probability,ledoux2013probability}.

\begin{figure}
    \centering
    \begin{tikzpicture}
  \node[draw, rounded corners, rectangle, minimum width=1.2cm, minimum height=1cm] (A) at (0,0) {Dudley's Integral};
  \node[draw, rounded corners, rectangle, minimum width=1.2cm, minimum height=1cm] (B) at (8,0) {Majorizing Measure Theorem};

  \node[draw, rounded corners, rectangle, minimum width=1.2cm, minimum height=1cm] (C) at (4,-4) {Rate-Distortion Integral};

\draw[->, dashed,  shorten >=2pt, shorten <=2pt, transform canvas={yshift=3pt}] 
    (A) -- node[above] {stationary process} (B);
\draw[->, shorten >=2pt, shorten <=2pt, transform canvas={yshift=-3pt}] 
    (B) -- node[above] {~} (A);

\draw[->, thick, shorten >=2pt, shorten <=2pt, transform canvas={xshift=-11pt}] 
    (A) -- node[left] {lifting} (C);

\draw[->, thick, shorten >=2pt, shorten <=2pt] 
    (C) -- node[left] {~} (A);
  
\draw[->, thick, shorten >=2pt, shorten <=2pt] 
    (B) -- node[left] {lifting} (C);

\draw[->, thick, shorten >=2pt, shorten <=2pt, transform canvas={xshift=11pt}] 
    (C) -- node[right] {data processing} (B);
\end{tikzpicture}
    \caption{Arrows indicate implications of results and the associated techniques or assumptions. For Gaussian processes, the majorizing measure theorem and the rate-distortion integral are two-sided bounds,
    whereas Dudley's integral gives only an upper bound without stationarity.}
    \label{fig1}
\end{figure}

The characterization \eqref{e1} offers three advantages over the majorizing measure theorem.
First, tensorization often reduces its evaluation to a one-dimensional problem.
Second, its two-sided bound holds for every fixed $\mu$, rather than only for the optimal choice over $T$.
Third, a representation that is convex in one distribution and concave in the other yields equivalent formulas by the minimax theorem.

In statistical applications such as empirical risk minimization (ERM), sharp rates depend on the local complexity of the region where the random index $Z$ concentrates \cite{van1990estimating,wainwright2019high}.
The right-hand side of \eqref{e1} can often be bounded using the variance of $Z\sim\mu$ rather than the diameter of $T$, giving localized generalization bounds without slicing.

As an application, we consider empirical risk minimization over Lipschitz images of Sobolev ellipsoids with smoothness parameter $\beta>1/2$.
The rate-distortion argument gives a risk bound of order $n^{-2\beta/(2\beta+1)}$, with a coefficient of order at most $(2\beta-1)^{-1/2}$ as $\beta\downarrow1/2$.
Direct localization of Dudley's entropy integral gives a coefficient of order $(2\beta-1)^{-1}$ as $\beta\downarrow1/2$.
Classical localized complexity methods can also recover the better dependence: the spectral calculation for linear ellipsoids is given in \cite[Theorem~2.1 and its proof]{mendelson2003performance} and \cite[Section~3.7]{koltchinskii2011oracle}, and can be combined with Gaussian comparison and a covering argument for nonlinear Lipschitz images.
The application illustrates how rate-distortion bounds achieve localization directly, without slicing or using the majorizing measure theorem; see Remark~\ref{rem2}.

\medskip
\emph{Further related work.}
Talagrand's majorizing measure theorem \citep{talagrand1987regularity} provides a sharp characterization of the supremum of a Gaussian process in terms of the $\gamma_2$ functional.
Several equivalent formulations of $\gamma_2$, including admissible nets and tree-based constructions, have been developed; see \cite{talagrand1992simple,talagrand2014upper,van2014probability}.
These formulations admit information-theoretic interpretations: for example, labeled nets correspond to code lengths and are related to the Kraft-McMillan inequality \citep{maurer2010majorizing,chu2023majorizing}.
However, none of these works establish a representation of the form \eqref{e1}.

A closely related line of work studies information-theoretic generalization bounds based on the mutual information between a random index and the process \citep{russo2016controlling,xu_raginsky,dziugaite2017computing,asadi2018chaining,bu2020tightening,aminian2023information}, typified by inequalities of the form \begin{align}
\mathbb{E}[X_Z]\lesssim \sup_{t\in T}\|X_t\|_G\sqrt{I(Y; Z)},
\label{e_russ}
\end{align}
where $Y=(X_t)_{t\in T}$. 
These results refine maximal inequalities \citep{jiao2018generalizations} and can be further sharpened via chaining \citep{asadi2018chaining}, but differ from \eqref{e1} in that they control $\mathbb{E}[X_Z]$ through $I(Y;Z)$ rather than the rate-distortion function of $Z$, and do not rely on lifting. 
Theorem~\ref{thm_both} combines the rate-distortion and mutual-information bounds. 

After a preliminary version of this paper appeared on arXiv, van Handel \cite{van2025subgaussian} used the optimal-transport formulation of sub-Gaussian comparison in Corollary~\ref{cor_comp} to derive a convex-order version by duality.
This convex-order comparison was subsequently used by Hua, Song, and Tudose \cite{huasongtudose2026convexity} to represent sub-Gaussian vectors as sums of Gaussian vectors in their work resolving Talagrand's convexity conjecture.
The same comparison is a main ingredient in \cite{openai2026bernoulli}, a short proof of the Bernoulli theorem by an internal model at OpenAI.
The rate-distortion integral has also been used to establish two-sided bounds for entropic optimal transport \cite{liu2026}.
Liu and Zadik \cite{liuzadik2026bernoulli} adapted the rate-distortion integral and lifting approach to Bernoulli processes, obtaining a Bayesian proof of the Bernoulli theorem and a characterization for each prescribed index distribution.

\medskip
\emph{Organization.}
Section~\ref{sec_back} and Section~\ref{sec_main} present the background and main results.
Section~\ref{sec_finite} proves the rate-distortion integral, i.e., the upper bound in \eqref{e1}.
Section~\ref{sec_maj} proves the lower bound in \eqref{e1},
and 
Section~\ref{sec_new}
shows that it implies the majorizing measure theorem by applying the minimax inequality and the data-processing inequality.
Section~\ref{sec_erm} provides an example of calculating the supremum and ERM over the Sobolev ellipsoid using \eqref{e1}.

\section{Background}
\label{sec_back}

\emph{Notation.} All probability spaces are completed.
$\wedge$ and $\vee$ denote the min and max of two numbers, and $[x]_+:=x\vee 0$.
Given probability measures $P$ and $Q$ on a metric space $(T,d)$, 
$W_p(P,Q):=\inf_{P_{XY}\in\Pi(P,Q)}\mathbb{E}^{\frac1p}[d(X,Y)^p]$ denotes the Wasserstein distance, where $\Pi(P,Q)$ denotes the set of couplings.
For brevity, $\Pi(P):=\Pi(P,P)$.
Product measures are denoted by $P\times Q$ or $P^{\otimes N}$.
The diameter of a set $T$ is 
${\rm diam}(T):=\sup_{x,y\in T}d(x,y)$.
$\lesssim$ means less than or equal to up to a universal constant,
and $\asymp$ means both $\lesssim$ and $\gtrsim$.
We denote by $\ell_2$ the Hilbert space of square-summable infinite sequences, equipped with the norm $\|\cdot\|_2$, abbreviated as $\|\cdot\|$.
The empirical distribution of a sequence $x=(x(1),x(2),\dots,x(d))$ is $\widehat{P}_x:=d^{-1}\sum_{j=1}^d\delta_{x(j)}$, where $\delta_u$ denotes the unit point mass at $u$.
Given a set $T$ in $\mathbb{R}^d$, define $T-T:=\{t-s:t,s\in T\}$.

\subsection{Upper bounds for stochastic processes}
\label{sec21}

Classically, Dudley's integral provides upper bounds on a sub-Gaussian process in terms of covering numbers. 
Following \cite[Exercise~2.40]{vershynin2018high}, define
the sub-Gaussian variance of a random variable $X$ as
\begin{align}
{\rm var}_G(X):=
\inf\Bigl\{\sigma^2:
\mathbb{E}[e^{\lambda(X-\mathbb{E}X)}]
\le 
e^{\sigma^2\lambda^2/2}
\text{ for all }\lambda\in \mathbb{R}\Bigr\},
\end{align}
and define the \emph{exact sub-Gaussian norm} by 
\begin{align}
\|X\|_G^2={\rm var}_G(X)+(\mathbb{E}[X])^2.
\label{e_psi}
\end{align}
Let $(X_t)_{t\in T}$ be a centered sub-Gaussian process, and define the natural metric
\begin{align}
d(t,t'):=\|X_t-X_{t'}\|_G,\qquad \forall t,t'\in T.
\label{e_dsub}
\end{align}
This is a pseudometric; indices at zero distance may be identified. All upper bounds below also hold with any dominating metric $d$, that is, with $\|X_t-X_{t'}\|_G\le d(t,t')$.
For Gaussian lower bounds, $d$ is the natural metric.
Dudley's integral \cite{dudley2016vn} states that
\begin{align}
\mathbb{E}[\sup_{t\in T}X_t]
\le K\int_0^{\Delta}\sqrt{\ln \mathsf{N}(T,\lambda)}\,d\lambda,
\label{e_dudley}
\end{align}
where $K$ is a universal constant, $\mathsf{N}(T,\lambda)$ is the least number of closed $d$-balls of radius $\lambda$, with centers in $T$, needed to cover $T$, and $\Delta$ denotes the diameter of $T$.

Recall the following definition of stationary Gaussian processes
(see e.g.\ \cite{van2014probability,talagrand2014upper}):

\begin{definition}
\label{def_stat}
The centered Gaussian process $(X_t)_{t\in T}$ is called stationary if there exists a group $G$ acting on $T$ such that
\begin{enumerate}
\item $d(g(t),g(s))=d(t,s)$ for all $t,s\in T$, $g\in G$;
\item for every $t,s\in T$, there exists $g\in G$ such that $t=g(s)$.
\end{enumerate}
\end{definition}

For stationary Gaussian processes, Dudley's integral is sharp \citep{fernique1975regularite}:
\begin{align}
\int_0^{\Delta}\sqrt{\ln \mathsf{N}(T,\lambda)}\,d\lambda
\le K'\mathbb{E}[\sup_{t\in T}X_t]
\label{e_reverse}
\end{align}
for a universal constant $K'>0$.
In \cite{maurey} an explicit estimate of $K'\le 432$ was provided.
For general nonstationary Gaussian processes, define
\begin{align}
I_{\mu}(t):=\int_0^{\Delta}\sqrt{\ln\frac1{\mu(B(t,\lambda))}}\,d\lambda.
\label{e_imu}
\end{align}
Talagrand's majorizing measure theorem states that
\begin{align}
\gamma_2(T):=\inf_{\mu}\sup_{t\in T}I_{\mu}(t)
\label{e2}
\end{align}
is equivalent, up to universal constants, to $\mathbb{E}[\sup_{t\in T}X_t]$.
It is also known \citep[Sections~3.3--3.4, especially Propositions~3.3.1 and~3.4.1]{talagrand2014upper} that
\begin{align}
{\rm Fer}(T)&:=\sup_{\mu}\int_T I_{\mu}(t)d\mu(t),
\label{e1.3}
\\
\delta_2(T)&:=\sup_{\mu}\inf_{t\in T}I_{\mu}(t)
\label{e1.4}
\end{align}
are equivalent to $\gamma_2(T)$ up to constants.

A consequence of $\gamma_2$ is sub-Gaussian comparison:
if $(Y_t)_{t\in T}$ is a centered separable Gaussian process with natural metric $d$, and $(X_t)_{t\in T}$ is centered, separable, and satisfies $\|X_t-X_s\|_G\le d(t,s)$ for all $s,t\in T$, then
\begin{align}
\mathbb{E}[\sup_{t\in T}X_t]\lesssim \mathbb{E}[\sup_{t\in T}Y_t].
\label{e_compare}
\end{align}

\subsection{Information-theoretic measures}
We recall the information-theoretic definitions from \cite{thomas2006elements}, using natural logarithms.
The entropy of a discrete random variable $X\sim \mu$ is $H(X):=\sum_x \mu(x)\ln\frac1{\mu(x)}$.
The relative entropy between $\mu$ and $\nu$ on the same measurable space is $D(\mu\|\nu):=\int\ln\frac{d\mu}{d\nu}d\mu$ if $\mu\ll\nu$, and $D(\mu\|\nu):=+\infty$ otherwise.
Mutual information for $(X,Y)\sim P_{XY}$ is defined as $I(X;Y):=D(P_{XY}\|P_X\times P_Y)$.
Conditional versions of these quantities are defined by averaging over the random variable which is conditioned on. 
For example, 
\begin{align}
D(P_{\hat{Z}|Z}\|\mu|P_Z):=\int D(P_{\hat{Z}|Z=z}\|\mu) dP_Z(z),
\label{e_kl}
\end{align}
and $H(Y|X):=\int H(Y|X=x)dP_X(x)$, where
\[
H(Y|X=x):=\sum_y P_{Y|X=x}(y)\ln\frac1{P_{Y|X=x}(y)}.
\]

Given a probability measure $\mu$ on a metric space $(T,d)$ and $\sigma>0$, define
\begin{align}
\Pi_{\sigma}(\mu)
:=\{P_{Z\hat{Z}}: P_Z=P_{\hat Z}=\mu,\ \mathbb{E}[d^2(Z,\hat Z)]\le \sigma^2\}.
\end{align}

\begin{definition}
\label{def_rd}
Given $\mu$ on a metric space $(T,d)$, define the rate-distortion function
\begin{align}
r_{\mu}(\sigma):=\inf_{P_{UZ}} I(U;Z),
\end{align}
where the infimum is over probability measures $P_{UZ}$ on $T\times T$ satisfying $P_Z=\mu$ and $\mathbb{E}[d^2(U,Z)]\le \sigma^2$.
Define
\begin{align}
i_{\mu}(\sigma):=\inf_{P_{Z\hat Z}\in \Pi_{\sigma}(\mu)} I(Z;\hat Z).
\end{align}
\end{definition}

Observe that
\begin{align}
r_{\mu}(\sigma)
\le
i_{\mu}(\sigma)
\le
r_{\mu}(\sigma/2).
\label{e21}
\end{align}
The first inequality follows by dropping the constraint $P_U=\mu$ in the infimum. 
The second holds since given $\mathbb{E}[d^2(U,Z)]\le \sigma^2/4$, we can construct $\hat{Z}$ such that $\hat{Z}$ and $Z$ are independent and identically distributed conditioned on $U$. Then $\mathbb{E}[d^2(\hat{Z},Z)]
\le \mathbb{E}[(d(\hat{Z},U)+d(U,Z))^2]
\le \sigma^2$ and $I(Z;\hat{Z})\le I(U;Z)$ by the data-processing inequality: for a Markov chain $Z-U-\hat Z$, processing $U$ into $\hat Z$ cannot increase its mutual information with $Z$.
Due to \eqref{e21}, our results can be stated in terms of either $\inf_{P_{Z\hat{Z}}\in\Pi_{\sigma}(P_Z)}I(Z;\hat{Z})$ or the rate-distortion function (up to constant factors).

For $Z\sim \mu$, define
\begin{align}
\sigma_{\rm m}(\mu):=\inf_{t\in T}\sqrt{\mathbb{E}[d^2(Z,t)]},
\end{align}
abbreviated as $\sigma_{\rm m}$ when there is no confusion.
By definition, $r_{\mu}(\sigma)=0$ whenever $\sigma>\sigma_{\rm m}(\mu)$.

\section{Main results}
\label{sec_main}
\subsection{Rate-distortion integral}

\begin{definition}\label{def_comp}
Given a stochastic process $(X_t(\omega))_{t\in T}$ and a distribution $\mu$ on $T$, define the complexity
\begin{align}
\mathsf{w}(\mu):=\sup\mathbb{E}[X_Z],
\end{align}
where the supremum is over all couplings of the process and a measurable random index $Z$ satisfying $P_Z=\mu$, for which the expectation is well defined. Auxiliary randomization on an enlarged probability space is allowed, and a well-defined expectation may equal $+\infty$.
\end{definition}
Section~\ref{sec_meas} gives conditions for measurability of $X_Z$.
The quantity $\mathsf{w}(\mu)$ measures complexity for a prescribed distribution of the random index.

\begin{theorem}\label{thm1}
If $(X_t)_{t\in T}$ is a centered sub-Gaussian process on a countable set $T$, and $\mu$ is a distribution on $T$ for which $\sigma_{\rm m}(\mu)<\infty$, then
\begin{align}
\mathsf{w}(\mu)
\le
2K\int_0^{\sigma_{\rm m}}
\sqrt{i_{\mu}(\sigma)}\,d\sigma
\le
4K\int_0^{\sigma_{\rm m}}
\sqrt{r_{\mu}(\sigma)}\,d\sigma,
\label{e2.1}
\end{align}
where $K$ is the constant in \eqref{e_dudley}.
\end{theorem}

Theorem~\ref{thm1} is proved in Section~\ref{sec_finite} for finite $T$ and extended to countable $T$ in Appendix~\ref{app_count}. It will be shown in Lemma~\ref{lem_finite} that if the right side of \eqref{e2.1} is finite, then $\mathbb{E}[|X_Z|]<\infty$.

The lifting argument also gives the mutual-information refinement discussed in \eqref{e_russ}.

\begin{theorem}\label{thm_both}
Suppose that $(X_t)_{t\in T}$ is a centered sub-Gaussian process, $|T|<\infty$, and let $Y:=(X_t)_{t\in T}\in\mathbb{R}^{|T|}$.
Let $Z$ be a random variable on $T$ arbitrarily correlated with $Y$, and set $\mu:=P_Z$.
Then
\begin{align}
\mathbb{E}[X_Z]
\le
2K\int_0^{\sigma_{\rm m}}
\sqrt{i_{\mu}(\sigma)\wedge I(Y;Z)}\,d\sigma.
\label{e_thm2}
\end{align}
\end{theorem}

Theorem~\ref{thm_both} is proved in Section~\ref{sec_finite}.
Since $\sigma_{\rm m}\le \sup_{t,z\in T} d(z,t)\lesssim \sup_{t\in T}\|X_t\|_G$, Theorem~\ref{thm_both} implies \eqref{e_russ}.
In \cite{liu2026}, we further prove that the inequality in Theorem~\ref{thm_both} is sharp up to a universal constant in the case of Gaussian processes: 
the lifted set in the proof of Theorem~\ref{thm_both} is a random set that approximately satisfies permutation symmetry due to concentration, 
which allows the argument for stationary processes to apply approximately.

While Theorem~\ref{thm1} is stated for countable $T$, it suffices for our statistical applications due to the following approximation result.

\begin{proposition}\label{prop1}
Suppose that $(T,d)$ is separable and $(X_t)_{t\in T}$ is a centered sub-Gaussian process with almost surely continuous sample paths with respect to $d$. For a Borel random index $Z\sim\mu$ with $\sigma_{\rm m}(\mu)<\infty$, the following bound holds whenever $\mathbb E[X_Z]$ is well defined:
\begin{align}
\mathbb{E}[X_Z]
\le
4K\lim_{\epsilon\downarrow 0}
\sup_{\mu':\,W_{\infty}(\mu,\mu')\le \epsilon}
\int_0^{\sigma_{\rm m}(\mu')}
\sqrt{r_{\mu'}(\sigma)}\,d\sigma.
\label{e25}
\end{align}
\end{proposition}
The supremum in \eqref{e25} is over probability measures on $T$. If its right side is finite, then $X_Z$ is absolutely integrable. The separability and continuity assumptions used here are clarified in Section~\ref{sec_meas}.
The proof of Proposition~\ref{prop1} is in Appendix~\ref{sec_prop1}.

\subsection{Sharpness for Gaussian processes}
\label{sec32}

While Dudley's integral is not sharp in general, its information-theoretic counterpart is sharp.

\begin{theorem}
\label{thm2}
Suppose that $(X_t)_{t\in T}$ is a centered Gaussian process, $T$ is a finite set, and $\mu$ is an arbitrary distribution on $T$.
Then
\begin{align}
K'\mathsf{w}(\mu)
\ge
\int_0^\infty \sqrt{i_{\mu}(\sigma)}\,d\sigma
\ge
\int_0^{\sigma_{\rm m}(\mu)}\sqrt{r_{\mu}(\sigma)}\,d\sigma,
\label{e217}
\end{align}
where $K'$ is any universal constant such that \eqref{e_reverse} holds.
\end{theorem}

We prove Theorem~\ref{thm2} in Section~\ref{sec_maj}, using the sharpness of Dudley's integral for stationary Gaussian processes.
In Appendix~\ref{app_thm2} we also extend the result to countable $T$.

From Theorem~\ref{thm2} we derive the classical majorizing measure theorem, with an estimate of the constant:

\begin{corollary}\label{cor_maj0}
For any centered Gaussian process,
\begin{align}
\mathbb{E}[\sup_{t\in T}X_t]
\ge
\bar{c}\inf_{\mu}\sup_{t\in T}I_{\mu}(t),
\label{e_maj0}
\end{align}
where $\bar{c}:=\sup_{a>1}\frac{\sqrt{1-a^{-2}}}
{2aK'+\sqrt{2\pi h(a^{-2})}}$, and $h(\epsilon):=\epsilon\ln\frac1{\epsilon}+(1-\epsilon)\ln\frac1{1-\epsilon}$ denotes the binary entropy function.
\end{corollary}
Here the supremum of the process is interpreted as the lattice supremum, following \cite[Section~2.2, equation~(2.2)]{talagrand2014upper}; thus
\[
\mathbb E\left[\sup_{t\in T}X_t\right]
:=\sup_{\substack{F\subseteq T\\0<|F|<\infty}}
\mathbb E\left[\max_{t\in F}X_t\right].
\]
The proof of Corollary~\ref{cor_maj0} can be found in Section~\ref{sec_new}.
Theorem~\ref{thm2} also gives sub-Gaussian comparison for each prescribed distribution of the random index.

\begin{corollary}\label{cor_comp}
Suppose that $(X_t(\omega))_{t\in T}$ is a centered sub-Gaussian process on a countable $T$,
with complexity function $\mathsf{w}_1$ and natural metric $d_1$.
Suppose that $(G_t(\omega'))_{t\in T}$ is a centered Gaussian process with complexity function $\mathsf{w}_2$ and natural metric $d_2$, and that $d_1$ is dominated by $d_2$.
Then
$
\mathsf{w}_1(\mu)\lesssim \mathsf{w}_2(\mu)
$
for any $\mu$ on $T$.
\end{corollary}

For finite $T$, Corollary~\ref{cor_comp} follows directly from Theorems~\ref{thm1} and~\ref{thm2}, since increasing the metric increases the rate-distortion function.
For countable $T$, choose $t_0$ with $\mu(t_0)>0$ and exhaust $T$ by finite sets $T_k\ni t_0$. If $\sum_t\mu(t)d_2(t,t_0)=\infty$, Appendix~\ref{app_thm2} gives $\mathsf w_2(\mu)=\infty$.
Otherwise set $Z_k=Z$ on $\{Z\in T_k\}$ and $Z_k=t_0$ on its complement, and let $\mu_k=P_{Z_k}$. Every random index for the Gaussian process with law $\mu_k$ can be changed to law $\mu$ by redistributing the excess mass at $t_0$ to $T_k^c$, using auxiliary randomization. Cauchy--Schwarz bounds the absolute change in the expectation of the selected process value by
\[
\frac{\sum_{t\notin T_k}\mu(t)d_2(t,t_0)}{\sqrt{\mu(t_0)}}\longrightarrow0.
\]
Thus $\limsup_k\mathsf w_2(\mu_k)\le\mathsf w_2(\mu)$. Apply the finite comparison to $X_{Z_k}$ and pass to the limit. Whenever the selected process value has a well-defined expectation, the truncated positive and negative parts converge separately, while the contribution from the reference point tends to zero. Taking the supremum proves the claim.
Corollary~\ref{cor_comp} implies the standard Gaussian comparison \eqref{e_compare}.
Van Handel \cite{van2025subgaussian} used Corollary~\ref{cor_comp} to derive, by duality, a convex-order comparison for sub-Gaussian random vectors.

The following Euclidean special case is used in Section~\ref{sec_erm}.

\begin{corollary}\label{cor1}
If $T$ is a countable subset of $\mathbb{R}^N$, a random index $Z$ takes values in $T$ and has covariance matrix $\Sigma$, and the natural metric of a centered sub-Gaussian process $(X_t)_{t\in T}$ is dominated by the Euclidean metric, then
$
\mathbb{E}[X_Z]\lesssim \tr(\sqrt{\Sigma})$.
\end{corollary}
The proof can be found in Section~\ref{sec_c3pf}.

\subsection{Application to ERM over the Sobolev ellipsoid}
\label{sec34}

Consider the mean estimation problem, where $T$ is a closed subset of $\ell_2$, and $X_1,\dots,X_n$ are i.i.d.\ random vectors in $\ell_2$ whose mean is an unknown $m\in T$.
We assume that $\|X_1-m\|_G\le 1$, where for a random vector $Y\in\ell_2$,
\begin{align}
\|Y\|_G:=\sup_{v:\|v\|_2=1}\|\langle Y,v\rangle\|_G.
\end{align}

Define the empirical risk
\begin{align}
\hat L(t):=\frac1n\sum_{i=1}^n \|t-X_i\|_2^2,
\label{e_48}
\end{align}
the population excess risk $L(t)-L(m):=\mathbb{E}[\|t-X_1\|^2-\|m-X_1\|^2]=\|t-m\|^2$, and the empirical risk minimizer
\begin{align}
\hat m:=\argmin_{t\in T}\hat L(t).
\end{align}
For the compact constraint sets in Theorem~\ref{thm_ellip}, a measurable minimizer exists. Defining the excess risk through the difference of losses avoids requiring $\mathbb E\|X_1\|^2<\infty$. Then
\begin{align}
\|\hat{m}-m\|^2
&=L(\hat{m})-L(m)
\\
&\le -\hat{L}(\hat{m})
+\hat{L}(m)
+L(\hat{m})-L(m)
\\
&=\frac1{\sqrt{n}}\chi_{m-\hat{m}}
\label{e312}
\end{align}
where $\chi_z:=-2\sqrt{n}\left<z,\frac1{n}\sum_{i=1}^nX_i-m\right>$,  and $(\frac1{2}\chi_z)_{z\in T-T}$ is a centered sub-Gaussian process, whose natural metric is dominated by the $\ell_2$ metric.
When $\mathbb E\|X_1\|^2<\infty$, the expectation of the right side of \eqref{e312} equals the expected generalization error $\mathbb E[L(\hat m)-\hat L(\hat m)]$. Without this additional moment assumption, \eqref{e312} still bounds the excess risk as defined above.

A general ellipsoid is defined as
\begin{align}
\mathcal{E}
:=
\left\{t:\sum_{i=1}^\infty \frac{t_i^2}{a_i^2}\le 1\right\},
\label{e31}
\end{align}
and a
Sobolev ellipsoid is the special case where $a_i=i^{-\beta}$ for some $\beta>0$.
Let $T=\phi(\mathcal E)$ be the image of $\mathcal E$ under a $1$-Lipschitz map $\phi:\ell_2\to\ell_2$.

\begin{theorem}\label{thm_ellip}
Under the mean-estimation assumptions above, let $T=\phi(\mathcal E)$, where $a_i=i^{-\beta}$, $\beta>1/2$, and $\phi:\ell_2\to\ell_2$ is $1$-Lipschitz. There exists a function $c(\beta)>0$ satisfying
$
\lim_{\beta\downarrow 1/2}\sqrt{2\beta-1}\,c(\beta)<\infty
$,
such that
$
\mathbb{E}[\|m-\hat m\|_2^2]
\le
c(\beta)n^{-2\beta/(2\beta+1)}$.
\end{theorem}

We prove Theorem~\ref{thm_ellip} in Section~\ref{sec73}.
\begin{remark}\label{rem2}
The referee proposed an alternative proof using classical arguments, which we briefly outline here. Details can be completed with AI assistance.
The linear spectral calculation follows by Cauchy--Schwarz; see \cite[Theorem~2.1 and its proof]{mendelson2003performance}, \cite[Section~3.7]{koltchinskii2011oracle}, and \cite{wei2020gauss}.
For nonlinear Lipschitz images, intersect an $r/2$-cover of $\mathcal E$ with $S_r=\{u\in\mathcal E:\|\phi(u)-m\|\le r\}$ and recenter each nonempty intersection at one of its points. The resulting at most $\mathsf N(\mathcal E,r/2)$ translated sets lie in $2\mathcal E\cap rB$, with $B$ the unit ball of $\ell_2$.
Sudakov--Fernique comparison within these sets and Gaussian concentration for their maximum bound the Gaussian width, using separable versions. Sub-Gaussian comparison and the moment bound \cite[Theorem~3.2]{dirksen2015tail} yield, uniformly for $1/2<\beta\le1$ and $r>0$,
\[
\begin{aligned}
\left\|\sup_{t\in T:\,\|t-m\|\le r}\frac12\chi_{m-t}\right\|_{L_2}
&\lesssim
\left(\sum_i\min\{a_i^2,r^2\}\right)^{1/2}
+r\sqrt{\ln\mathsf N(\mathcal E,r/2)}
\\
&\lesssim (2\beta-1)^{-1/2}r^{1-1/(2\beta)}.
\end{aligned}
\]
The entropy bound $\ln\mathsf N(\mathcal E,\varepsilon)\lesssim\varepsilon^{-1/\beta}$ is uniform in this $\beta$-range, so only the spectral sum contributes the factor $(2\beta-1)^{-1/2}$. The Gaussian width absorbs the additional radius term because $0\in(T-m)\cap rB$. Slicing and \eqref{e312} give the risk bound.
Our information-theoretic proof in Section~\ref{sec73} uses the rate-distortion integral without slicing or the majorizing measure theorem. Lemma~\ref{lem_2} makes explicit that $1$-Lipschitz maps do not increase the localized complexity bound, since they cannot increase distortion or mutual information.
\end{remark}

\section{Proof of the rate-distortion integral}
\label{sec_finite}

In this section we prove Theorem~\ref{thm1} and Theorem~\ref{thm_both}.

\subsection{Measure-theoretic remarks}
\label{sec_meas}
We first clarify the measurability of $X_Z$ in
Definition~\ref{def_comp}.
First, if $T=\{t_0,t_1,\dots\}$ is countable,
we assume that for each $t\in T$, $X_t(\omega)$ is a measurable function of $\omega\in\Omega$ (where $\Omega$ denotes the sample space).  
Suppose that $Z(\omega)\in\{t_0,t_1,\dots\}$ is also a measurable function of $\omega\in\Omega$.
Then $X_{Z(\omega)}(\omega)$ is a measurable function of $\omega$, because for any Borel set $B\subseteq \mathbb{R}$,
\begin{align}
\{\omega\colon X_{Z(\omega)}(\omega)\in B\}
=\bigcup_{i=0}^{\infty}\{\omega\colon 
Z(\omega)=t_i\}\cap\{\omega\colon X_{t_i}(\omega)\in B\}
\label{e_meas}
\end{align}
which is measurable.

If $T$ is uncountable, this argument does not establish measurability because the union in \eqref{e_meas} is uncountable.
Separability of a process means that its relevant suprema can be computed on a fixed countable subset outside one null event; it does not imply continuity of sample paths.
For Proposition~\ref{prop1} we use the explicit assumptions that $(T,d)$ has a countable dense subset $T_0=\{t_0,t_1,\dots\}$ and that sample paths are $d$-continuous on one event of probability one.
Under these assumptions, $X_{Z(\omega)}(\omega)$ is measurable for every Borel measurable $Z$.
Define the random variable $Z_k:=t_i$ where 
\begin{align}
i=\argmin_{j\in\{0,1,\dots,k\}}d(Z(\omega),t_j)
\label{e48}
\end{align}
with ties broken by the smallest index.
Assuming measurability of $Z(\omega)$, we have measurability of $d(Z(\omega),t_j)$ for each $j$, so that $Z_k$, and hence $X_{Z_k}$, is measurable.
Density of $T_0$ ensures $\lim_{k\to\infty}Z_k=Z$ almost surely, and continuity implies that $X_Z$ agrees almost surely with the measurable limit $\lim_{k\to\infty}X_{Z_k}$. Completeness of the probability space therefore implies measurability of $X_Z$.

\subsection{Lifting via type classes}
We first prove Theorem~\ref{thm1} in the case of finite $T$;
extension beyond the finite case follows by a limiting argument, and is deferred to  Appendix~\ref{sec_countable}.

Let $n:=|T|$. 
Without loss of generality, we can identify $T$ with the standard basis vectors $e_1,e_2,\dots,e_n$ of $\mathbb{R}^n$, and define $Y_i:=X_{e_i}$ for each $i\in\{1,\dots,n\}$.
Then the process has the inner-product representation 
\begin{align}
X_z=\left<Y,z\right>,\quad\forall z\in T.
\label{e_rep}
\end{align}
The inner-product representation simplifies the notation but is not needed to prove Theorem~\ref{thm1}.
Now let $Y^N\sim P_Y^{\otimes N}$.
Recall the definitions related to types in Appendix~\ref{sec_23}. 
For any $\mu$ on $T$, define
\begin{align}
\mathcal{I}:=\{N\in \mathbb{Z}\colon \textrm{$N>0$, $N\mu$ is integer}\} 
\label{e_ni}
\end{align}
which is nonempty if $\mu$ is rational.
For $N\in\mathcal{I}$, define the type class
\begin{align}
\mathcal{C}_{\mu}
&:=\{z^N\in T^N\colon \widehat{P}_{z^N}=\mu\}.
\label{e_c1}
\end{align}
The following two lemmas show that a coupling can be approximated by a matching:

\begin{lemma}\label{lem_von}
Let $\mu$ be a rational distribution on $\mathbb{R}^n$, and let $N\in\mathcal{I}$. 
For any $y^N\in \mathbb{R}^{nN}$, we have
\begin{align}
\frac1{\sqrt{N}}\min_{s^N\in\mathcal{C}_{\mu}}\|y^N-s^N\|_2=W_2(\widehat{P}_{y^N},\mu).
\end{align}
\end{lemma}
\begin{lemma}\label{lem:cycle-rounding}
Suppose that $\mathcal{X}$ and $\mathcal{Y}$ are finite sets,
$N>0$ is an integer,
and $P_{XY}$ is a distribution with the property that both $NP_X$ and $NP_Y$ are integer.
Then there exists $Q_{XY}$ satisfying $Q_X=P_X$, $Q_Y=P_Y$, $\max_{x,y}|Q_{XY}(x,y)-P_{XY}(x,y)|\le \frac1{N}$,
and that $NQ_{XY}$ is integer.
\end{lemma}
Lemma~\ref{lem_von} has no approximation error, 
whereas 
Lemma~\ref{lem:cycle-rounding} approximates the distribution itself, with an error of $1/N$.
For this proof, Lemma~\ref{lem:cycle-rounding} alone would suffice, since the $1/N$ error vanishes in the limit.
The proofs of both lemmas can be found in Appendix~\ref{sec_mc}.

The next lemma relates the lifted maximum to the original complexity:
\begin{lemma}\label{lem_type}
For any rational $\mu$, 
we have 
\begin{align}
\lim_{\mathcal{I}\ni N\to\infty}\frac1{N}\mathbb{E}
[\max_{z^N\in\mathcal{C}_{\mu}}\left<Y^N,z^N\right>]
=\mathsf{w}(\mu),
\end{align}
where $\lim_{\mathcal{I}\ni N\to\infty}$ denotes the limit along the subsequence in $\mathcal{I}$. 
\end{lemma}
\begin{proof}
By Lemma~\ref{lem_von}, we see that for any $N\in\mathcal{I}$,
\begin{align}
\frac1{N}\max_{z^N\in\mathcal{C}_{\mu}}\left<y^N,z^N\right>
&=\frac{\mathbb{E}_{\widehat{P}_{y^N}}[\|Y\|_2^2]
+\mathbb{E}_{\mu}[\|Z\|_2^2]
-W_2^2(\widehat{P}_{y^N},\mu)}
{2}.
\end{align}
The empirical measures on $\mathbb R^n$ converge weakly to $P_Y$ almost surely, and the strong law gives convergence of their second moments since $\mathbb E\|Y\|_2^2<\infty$. These two properties imply almost-sure convergence in $W_2$ (cf.\ \cite{chewi2024statistical}). Choose a deterministic sequence $\delta_N\downarrow0$ slowly enough that $\mathbb P[W_2(P_Y,\widehat P_{Y^N})\le\delta_N]\to1$.

Define $M_N:=\sup_{P_{Y'}\colon W_2(P_Y,P_{Y'})\le \delta_N}
\frac{\mathbb{E}_{P_{Y'}}[\|Y\|_2^2]
+\mathbb{E}_{\mu}[\|Z\|_2^2]
-W_2^2(P_{Y'},\mu)}
{2}$.
The triangle inequality for $W_2$ shows that $\mathbb{E}_{P_{Y'}}[\|Y\|_2^2]
-W_2^2(P_{Y'},\mu)$ is continuous in $P_{Y'}$ with respect to $W_2$, and hence
\begin{align}
\lim_{\mathcal{I}\ni N\to\infty}
M_N
&=
\sup_{P_{YZ}\in\Pi(P_Y,\mu)}\mathbb{E}[\left<Y,Z\right>]
\\
&=\mathsf{w}(\mu)
.
\label{e_lem60}
\end{align}
Define
\begin{align}
\mathcal{D}
&:=\{y^N\in\mathbb{R}^{nN}\colon
W_2(P_Y, \widehat{P}_{y^N})\le \delta_N
\}.
\end{align}
By the choice of $\delta_N$, we have $\lim_{\mathcal{I}\ni N\to\infty}\mathbb{P}[Y^N\in \mathcal{D}]=1$.
Therefore, 
\begin{align}
\limsup_{\mathcal{I}\ni  N\to\infty}\frac1{N}\mathbb{E}
[\max_{z^N\in\mathcal{C}_{\mu}}\left<Y^N,z^N\right>1_{Y^N\in\mathcal{D}}]
&=\limsup_{\mathcal{I}\ni  N\to\infty}
\mathbb{E}[
\sup_{P_{YZ}\in\Pi(\widehat{P}_{Y^N},\mu)}
\mathbb{E}[\left<Y,Z\right>]
1_{Y^N\in\mathcal{D}}
]
\label{e_lem63}
\\
&\le 
\lim_{\mathcal{I}\ni  N\to\infty}M_N\mathbb{E}
[1_{Y^N\in\mathcal{D}}]
=
\mathsf{w}(\mu),
\label{e_lem61}
\end{align}
where the first equality follows from Lemma~\ref{lem_von}.
Similarly, if the $\sup$ in the definition of $M_N$ is replaced by $\inf$,
then \eqref{e_lem60} still holds, but the inequality sign in \eqref{e_lem61} will be reversed, and the $\limsup$ in \eqref{e_lem63} will be replaced by $\liminf$.
Hence we have 
\begin{align}
\lim_{\mathcal{I}\ni  N\to\infty}\frac1{N}\mathbb{E}
[\max_{z^N\in\mathcal{C}_{\mu}}\left<Y^N,z^N\right>1_{Y^N\in\mathcal{D}}]
=
\mathsf{w}(\mu).
\label{e_lem62}
\end{align}

It remains to show that $\lim_{\mathcal{I}\ni  N\to\infty}\frac1{N}\mathbb{E}
[\max_{z^N\in\mathcal{C}_{\mu}}\left<Y^N,z^N\right>1_{Y^N\in\mathcal{D}^c}]
=0$,
which, by the Cauchy-Schwarz inequality, would follow once we show
\begin{align}
\limsup_{\mathcal{I}\ni  N\to\infty}\mathbb{E}
\left[
\left(
\frac1{N}\max_{z^N\in\mathcal{C}_{\mu}}\left<Y^N,z^N\right>
\right)^2
\right]
<\infty.
\label{e_l61}
\end{align}
To show \eqref{e_l61}, use the standard-basis representation of $T$ to obtain
\[
\left|\frac1N\max_{z^N\in\mathcal C_\mu}\langle Y^N,z^N\rangle\right|
\le \frac1N\sum_{j=1}^N\|Y_j\|_2.
\]
By Jensen's inequality, the second moment of the right side is at most $\mathbb E\|Y\|_2^2<\infty$, independently of $N$. This proves \eqref{e_l61} and completes the uniform-integrability argument.
\end{proof}

\subsection{Covering numbers and mutual information}

We use the following bound for sums of independent sub-Gaussian random variables; see \cite[Exercise~2.40]{vershynin2018high}:
\begin{proposition}\label{prop_1}
Suppose that $X_1,X_2,\dots,X_n$ are zero-mean, independent random variables, and $\|X_i\|_G\le K_i$, $i=1,2,\dots,n$.
Then
\[
\|X_1+X_2+\dots+X_n\|_G\le \sqrt{K_1^2+K_2^2+\dots+K_n^2}.
\]
\end{proposition}

By Proposition~\ref{prop_1}, $(\left<Y^N,z^N\right>)_{z^N\in\mathcal{C}}$ is a sub-Gaussian process with respect to the following metric on $T^N$:
\begin{align}
d_N(z^N,\hat{z}^N):=\sqrt{\sum_{i=1}^N d^2(z_i,\hat{z}_i)}.
\label{e_dn}
\end{align}
We have the following estimate of the covering number $\mathsf{N}(\mathcal{C},\cdot)$ under the metric $d_N$:

\begin{lemma}\label{lem_nc}
Let $\mu$ be rational, and $\mathcal{C}:=\mathcal{C}_{\mu}$.
For any $\sigma>0$, we have
\begin{align}
\limsup_{\mathcal{I}\ni N\to\infty}
\frac1{N}\ln \mathsf{N}(\mathcal{C},2\sqrt{N}\sigma)
&\le 
i_{\mu}(\sigma);
\label{e37_1}
\\
\liminf_{\mathcal{I}\ni N\to\infty}\frac1{N}\ln \mathsf{N}(\mathcal{C},2\sqrt{N}\sigma)
&\ge 
i_{\mu}(2\sigma).
\label{e37_2}
\end{align}
\end{lemma}
\begin{proof}
All balls in this proof are balls in $(\mathcal C,d_N)$, so they are intersected with $\mathcal C$. For any $t^N\in \mathcal C$, we have
\begin{align}
\mu^{\otimes N}(B(t^N,\sqrt{N}\sigma))
=\sum_{P_{Z\hat{Z}}\in\mathcal{S}_N}
\mu^{\otimes N}(\mathcal{C}_{P_{Z\hat{Z}}}(t^N))
\end{align}
where $\mathcal{S}_N$ denotes the set of all $P_{Z\hat{Z}}\in\Pi_{\sigma}(\mu)$ such that $NP_{Z\hat{Z}}$ is integer,
and $\mathcal{C}_{P_{Z\hat{Z}}}(t^N):=\{u^N\in{T}^N\colon \textrm{$(t^N,u^N)$ has type $P_{Z\hat{Z}}$}\}$.
%
We have
\begin{align}
\max_{P_{Z\hat{Z}}\in\mathcal{S}_N}
\mu^{\otimes N}(\mathcal{C}_{P_{Z\hat{Z}}}(t^N))
\le 
\sum_{P_{Z\hat{Z}}\in\mathcal{S}_N}
\mu^{\otimes N}(\mathcal{C}_{P_{Z\hat{Z}}}(t^N))
\le 
|\mathcal{S}_N|\max_{P_{Z\hat{Z}}\in\mathcal{S}_N}
\mu^{\otimes N}(\mathcal{C}_{P_{Z\hat{Z}}}(t^N))
\label{e_lem51}
\end{align}
where $\lim_{\mathcal{I}\ni N\to\infty}\frac1{N}\ln|\mathcal{S}_N|=0$ by 
Lemma~\ref{lem_ck}.
By
\eqref{e_ctype}, we have
\begin{align}
\left|\frac1{N}\ln\mu^{\otimes N}(\mathcal{C}_{P_{Z\hat{Z}}}(t^N))
+I(Z;\hat{Z})
\right|\le 
c_N
\label{e_lem52}
\end{align}
for any $P_{Z\hat{Z}}\in\mathcal{S}_N$,
where $c_N:=\frac{2|{T}|^2}{N}\ln(N+1)$.
In particular, this implies that 
\begin{align}
\left|\frac1{N}\ln\max_{P_{Z\hat{Z}}\in\mathcal{S}_N}\mu^{\otimes N}(\mathcal{C}_{P_{Z\hat{Z}}}(t^N))
+\min_{P_{Z\hat{Z}}\in\mathcal{S}_N}I(Z;\hat{Z})
\right|\le 
c_N.
\label{e_lem52_1}
\end{align}
Next, we show that 
\begin{align}
\lim_{\mathcal{I}\ni N\to\infty}\min_{P_{Z\hat{Z}}\in\mathcal{S}_N}
I(Z;\hat{Z})
=
i_{\mu}(\sigma).
\label{e_imu1}
\end{align}
The $\ge$ inequality follows by restricting the set of admissible couplings.
For the $\le$ part, 
let $\tau\in(0,\sigma/2)$ be arbitrary, 
and let $P_{Z\hat{Z}}^*$ be the minimizer of $I(Z;\hat{Z})$ in $\Pi_{\sigma-\tau}(\mu)$ (i.e., the optimal distribution in the definition of $i_{\mu}(\sigma-\tau)$).
By Lemma~\ref{lem:cycle-rounding},
for all sufficiently large $N\in\mathcal I$, there exists $P^N_{Z\hat{Z}}\in\mathcal{S}_N$ such that $\max_{z,u}|P_{Z\hat{Z}}^*(z,u)-P^N_{Z\hat{Z}}(z,u)|\le \frac1{N}$.
Therefore $\lim_{\mathcal{I}\ni N\to\infty}I(Z;\hat{Z})=i_{\mu}(\sigma-\tau)$ and $\lim_{\mathcal{I}\ni N\to\infty}\mathbb{E}[d^2(Z,\hat{Z})]\le(\sigma-\tau)^2$ where $(Z,\hat{Z})\sim P^N_{Z\hat{Z}}$.
Thus the left side of \eqref{e_imu1} is bounded above by $i_{\mu}(\sigma-\tau)$.
This establishes the $\le$ part of \eqref{e_imu1} by taking $\tau\downarrow 0$ and using the continuity of $i_{\mu}(\cdot)$ on finite $T$.
Now combining \eqref{e_lem51}, \eqref{e_lem52_1} and \eqref{e_imu1}, we obtain
\begin{align}
\left|\frac1{N}\ln\frac1{\mu^{\otimes N}(B(t^N,\sqrt{N}\sigma))}
-i_{\mu}(\sigma)
 \right|
\le c_N'
\label{e_38}
\end{align}
where $c_N'>0$ is a vanishing sequence possibly depending on $\sigma$ and $\mu$.
Let ${\sf P}(\mathcal{C},2\sqrt{N}\sigma)$ denote the packing number, i.e.\ the maximal number of points in $\mathcal{C}$ separated by distance strictly greater than $2\sqrt{N}\sigma$;
then $\mathsf{N}(\mathcal{C},2\sqrt{N}\sigma)\le {\sf P}(\mathcal{C},2\sqrt{N}\sigma)$ (see e.g.\ \cite{van2014probability}),
and standard volume estimates yield ${\sf P}(\mathcal{C},2\sqrt{N}\sigma)\le \frac1{\min_{t^N\in \mathcal{C}}\mu^{\otimes N}(
B(t^N,\sqrt{N}\sigma))}$, proving \eqref{e37_1}.
Moreover, by a volume argument we also have  $\mathsf{N}(\mathcal{C},2\sqrt{N}\sigma)\cdot\max_{t^N\in \mathcal{C}}\mu^{\otimes N}(B(t^N,2\sqrt{N}\sigma))
\ge \mu^{\otimes N}(\mathcal{C})$.
By Lemma~\ref{lem_ck}, we have $\lim_{N\to\infty}\frac1{N}\ln\mu^{\otimes N}(\mathcal{C})=0$.
Then using \eqref{e_38} with $\sigma$ replaced by $2\sigma$, we obtain  \eqref{e37_2}.
\end{proof}

\subsection{Proof of Theorem~\ref{thm1}}

We present here the proof of Theorem~\ref{thm1} for finite $T$,
and defer the extension to countable $T$ to  Appendix~\ref{sec_countable}.
When $|T|<\infty$, it suffices to consider rational $\mu$. To justify the approximation, couple $\mu_j\to\mu$ maximally, so that $\eta_j^2:=\mathbb E d^2(Z_j,Z)\to0$. Applying the coupling kernels independently to both coordinates of a self-coupling gives
\[
i_{\mu_j}(\sigma+2\eta_j)\le i_\mu(\sigma),
\qquad i_\mu(\sigma+2\eta_j)\le i_{\mu_j}(\sigma).
\]
Since $D\mapsto i_\mu(\sqrt D)$ is finite and convex, it is continuous for $D>0$. The bound $i_{\mu_j}\le\ln|T|$, together with continuity of $\sigma_{\rm m}(\mu_j)$, gives convergence of the integrals. The same coupling changes any selected expectation by at most
\[
\bigl(\mathbb E\max_{s,t\in T}|X_s-X_t|^2\bigr)^{1/2}
\mathbb P[Z_j\ne Z]^{1/2},
\]
so $\mathsf w(\mu_j)\to\mathsf w(\mu)$.
Applying Dudley's integral \eqref{e_dudley} to $\mathcal{C}:=\mathcal{C}_{\mu}$ yields
\begin{align}
\mathbb{E}[\max_{z^N\in\mathcal{C}}
\left<Y^N,z^N\right>]
\le K\int_0^{\Delta_{\rm m}}
\sqrt{\ln \mathsf{N}(\mathcal{C},2\sqrt{N}\sigma)}
2\sqrt{N}\,d\sigma
\label{e36}
\end{align}
where $\Delta_{\rm m}:=\sup_{N\in\mathcal I}\frac{{\rm diam}(\mathcal C)}{2\sqrt N}\le\sigma_{\rm m}$.
Indeed, for every $t\in T$ and $z^N\in\mathcal C$,
\[
d_N(z^N,t^N)^2=N\mathbb E_\mu d(Z,t)^2.
\]
The triangle inequality and an infimum over $t$ give the claim.
Now multiply \eqref{e36} by $1/N$, take $\mathcal{I}\ni N\to\infty$, and apply Lemma~\ref{lem_nc} and Lemma~\ref{lem_type}.
We can apply the dominated $\limsup$ inequality,
since 
$
\frac1{N}\sqrt{\ln \mathsf{N}(\mathcal{C},2\sqrt{N}\sigma)}
2\sqrt{N}\le \frac2{\sqrt{N}}\sqrt{\ln(n^N)}=2\sqrt{\ln n}
$, which is independent of $N$ and absolutely integrable on $\sigma\in [0,\Delta_{\rm m}]$.
Then Theorem~\ref{thm1} for finite $T$ follows from Lemma~\ref{lem_nc},  Lemma~\ref{lem_type}, and \eqref{e21}.

\subsection{Proof of Theorem~\ref{thm_both}}
As in the proof of Theorem~\ref{thm1}, we can identify $T=\{e_1,\dots,e_n\}$ as the standard basis vectors, so that $X_z=\left<Y,z\right>$.
Furthermore, we assume that $|\mathcal{Y}|<\infty$ and $P_{YZ}$ is rational; 
the general case can then be proved by a limiting argument, as follows. For general $Y$, replace it by $Y_j=\mathbb E[Y\mid\mathcal F_j]$, where $(\mathcal F_j)$ are increasing finite sigma-fields generating $\sigma(Y)$. Conditional Jensen's inequality preserves all centered sub-Gaussian increment bounds, $I(Y_j;Z)\le I(Y;Z)$, and $\mathbb E\langle Y_j,Z\rangle\to\mathbb E\langle Y,Z\rangle$. For a finite-valued $Y_j$, approximate the joint probabilities by rational probabilities on the same support and center each process coordinate. The expected process values, mutual informations, and exact sub-Gaussian norms of the finitely many increments converge. The last assertion follows directly from the finite-support moment generating functions: convergence is uniform for bounded Laplace parameters, and bounded support controls the remaining parameters. Thus a factor tending to one in the controlling metric suffices, and passing to the limit proves the general case.
Define 
\begin{align}
\mathcal{I}:=\{N\in\mathbb{Z}\colon \textrm{$N>0$, $NP_{YZ}$ is integer}\}.
\end{align}
For each $N\in\mathcal{I}$, let $\mathcal{C}_{\mu}$ denote the type class of $\mu$, consisting of all sequences of type $\mu=P_Z$.
Select 
\begin{align}
L:=\lfloor\exp(NI(Y;Z)+N^{1/3})\rfloor
\end{align} 
sequences from $\mathcal{C}_{\mu}$ uniformly at random with replacement, and call this random set $\mathcal{A}$. 
Let $Y^N\sim P_Y^{\otimes N}$ (independent of this random set).
Then we will show that 
\begin{align}
\mathbb{E}[\left<Y,Z\right>]
\le 
\liminf_{\mathcal{I}\ni N\to\infty}\mathbb{E}[\max_{z^N\in\mathcal{A}}\frac1{N}\left<Y^N,z^N\right>].
\label{e_66}
\end{align}
Indeed, set 
\begin{align}
\mathcal{B}&:=\{(y^N)\colon \widehat{P}_{y^N}=P_Y\};
\\
\mathcal{D}&:=\{(y^N,z^N)\colon \widehat{P}_{y^Nz^N}=P_{YZ}\},
\end{align}
and for any $y^N\in\mathcal{B}$, define $\mathcal{D}(y^N):=\{z^N\colon (y^N,z^N)\in\mathcal{D}\}$.
Standard results about type estimates (immediate from Lemma~\ref{lem_t2} and Lemma~\ref{lem_ck}) state that:
\begin{align}
\left|\frac1{N}\ln |\mathcal{D}(y^N)|
-H(Z|Y)
\right|
&\le 
C_1\frac{\ln N}{N},
\quad \forall y^N\in \mathcal{B},
\label{e_64}
\\
\left|\frac1{N}\ln|\mathcal{C}_{\mu}|
-H(Z)
\right|
&\le 
C_1\frac{\ln N}{N},
\end{align}
where $C_1,C_2\dots$ in this proof are positive constants independent of $N$.
Let $\pi$ be the projection from $\mathbb{R}^{nN}$ to $\mathcal{B}$ (i.e., map to the closest point in $\mathcal{B}$, with ties broken arbitrarily).
Define $\mathcal{E}$ as the event that $\mathcal{D}(\pi(Y^N))\cap \mathcal{A}=\emptyset$.
Then, \eqref{e_64} implies that
\begin{align}
\mathbb{P}[\mathcal{E}]
&=
\left(1-
|\mathcal{D}(y^N)|/
|\mathcal{C}_{\mu}|
\right)^L
\\
&\le \left(1-
\exp(-NI(Y;Z)-2C_1\ln N)
\right)^L
\\
&\le 
\exp(-C_2N^{1/3}),
\end{align}
where $y^N$ is any sequence in $\mathcal{B}$.
The definitions give 
\begin{align}
\max_{z^N\in\mathcal{A}}\frac1{N}\left<\pi(Y^N),z^N\right>1_{\mathcal{E}^c}
\ge \mathbb{E}[\left<Y,Z\right>]1_{\mathcal{E}^c}
\label{e_t21}
\end{align}
for all $N\in\mathcal{I}$ almost surely.
By our assumption of $|T|,|\mathcal{Y}|<\infty$, we have
\[
\left|\max_{z^N\in\mathcal{A}}\frac1{N}\left<\pi(Y^N),z^N\right>\right|
\le \max_{y\in\mathcal{Y},z\in T}|\left<y,z\right>|\le C_3.
\]
Therefore $\lim_{\mathcal{I}\ni N\to\infty}\mathbb{P}[\mathcal{E}]=0$ and the Cauchy-Schwarz inequality shows that 
\begin{align}
\limsup_{\mathcal{I}\ni N\to\infty}\mathbb{E}
\left[\left|\max_{z^N\in\mathcal{A}}\frac1{N}\left<\pi(Y^N),z^N\right>\right|\cdot 1_{\mathcal{E}}\right]
=0.
\end{align}
Then \eqref{e_t21} implies 
\begin{align}
\liminf_{\mathcal{I}\ni N\to\infty}
\mathbb{E}\left[\max_{z^N\in\mathcal{A}}\frac1{N}\left<\pi(Y^N),z^N\right>
\right]
\ge \mathbb{E}[\left<Y,Z\right>].
\end{align}
Furthermore,
\begin{align}
\max_{z^N\in\mathcal{A}}\frac1{N}\left<Y^N,z^N\right>
\ge 
\max_{z^N\in\mathcal{A}}\frac1{N}\left<\pi(Y^N),z^N\right>
-\max_{z^N\in\mathcal{A}}\frac1{N}\left<\pi(Y^N)-Y^N,z^N\right>
\end{align}
From Lemma~\ref{lem_von}, we have
\begin{align}
\max_{z^N\in\mathcal{A}}\frac1{N}\left<\pi(Y^N)-Y^N,z^N\right>
&\le 
\frac1{\sqrt{N}}\|\pi(Y^N)-Y^N\|_2\cdot
\frac1{\sqrt{N}}\max_{z^N\in\mathcal{A}}\|z^N\|_2
\\
&\le C_4W_2(P_Y,\widehat{P}_{Y^N})
\end{align}
for some $C_4>0$ independent of $N$ since $|T|<\infty$.
The almost-sure convergence of $W_2(P_Y,\widehat{P}_{Y^N})$ to zero also holds in expectation: $Y$ has finite support at this stage, so these distances are uniformly bounded \citep{chewi2024statistical}.
Thus \eqref{e_66} is verified by collecting the bounds above.

It remains to apply the Dudley inequality to the right side of \eqref{e_66}.
A maximal $2\sqrt N\sigma$-separated subset of $\mathcal A$ covers $\mathcal A$ at that radius. Its $\sqrt N\sigma$-balls in $\mathcal C$ are disjoint (using open balls, or an arbitrarily small change of radius), so the packing argument in Lemma~\ref{lem_nc} gives the same covering upper bound for $\mathcal A$. 
Also, $|\mathcal{A}|\le L$.
Therefore,
\begin{align}
\limsup_{\mathcal{I}\ni N\to\infty}\frac1{N}\ln \mathsf{N}(\mathcal{A},2\sqrt{N}\sigma)
\le
 i_{\mu}(\sigma)\wedge I(Y;Z).
\end{align}
The claim then follows again by applying the dominated $\limsup$ inequality (see the proof of Theorem~\ref{thm1}).

\section{Proof of sharpness of the rate-distortion integral}
\label{sec_maj}
\subsection{Proof of Theorem~\ref{thm2}}
We present a proof of Theorem~\ref{thm2}  using sharpness of Dudley's integral for stationary Gaussian processes.
We first consider the case where $n:=|T|<\infty$ and $\mu$ is rational.
Recall the representation of $X_z$ as an inner product 
in \eqref{e_rep},
which embeds $Y,Z$ in $\mathbb{R}^n$ where $Y$ is a normal vector,
and the definitions of $Y^N$, $\mathcal{I}$, and $\mathcal{C}:=\mathcal{C}_{\mu}$ as in \eqref{e_rep}, \eqref{e_ni}, \eqref{e_c1}.
Then $(\left<Y^N,z^N\right>)_{z^N\in\mathcal{C}}$ is a stationary Gaussian process with respect to the metric \eqref{e_dn}, according to Definition~\ref{def_stat},
due to invariance under the transitive group of coordinate permutations acting on $\mathcal{C}$.
By Lemma~\ref{lem_type} and the sharpness of Dudley's integral for stationary Gaussian processes \eqref{e_reverse}, we see that 
\begin{align}
K'{\sf w}(\mu)
&\ge 
\liminf_{\mathcal{I}\ni N\to\infty}
\frac1{N}\int_0^{\infty} \sqrt{\ln \mathsf{N}(\mathcal{C},\sqrt{N}\epsilon)} \sqrt{N}\,d\epsilon
\\
&\ge \int_0^{\infty}\sqrt{i_{\mu}(\epsilon)}d\epsilon
\label{e2170}  
\end{align}
where $\mathcal{I}$ is defined in \eqref{e_ni}, and the last step follows from 
Lemma~\ref{lem_nc} (with $\epsilon:=2\sigma$) and Fatou's Lemma.
The lower bound in terms of $r_{\mu}$ follows from \eqref{e21}.
This establishes the theorem for finite $T$ and rational $\mu$.
Next, we can drop the rationality assumption under $|T|<\infty$: 
Both $\mathsf{w}(\mu)$ and $\int_0^{\infty}\sqrt{i_{\mu}(\sigma)}d\sigma$ are continuous with respect to $\mu$ by the coupling argument in the proof of Theorem~\ref{thm1}. Thus the rational case implies the general case.
For the extension to countable $T$, see Appendix~\ref{app_thm2}.

\subsection{Proof of Corollary~\ref{cor1}}
\label{sec_c3pf}

It suffices to consider the Gaussian process \(G_t=\langle G,t\rangle\), \(G\sim \mathcal{N}(0,I_N)\), since the general sub-Gaussian case follows from Corollary~\ref{cor_comp}. Let \(Z\sim\mu\) and \(\Sigma=\operatorname{cov}(Z)\). Since \(G\) is centered,
$
\mathbb E\langle G,Z\rangle
=
\mathbb E\langle G,Z-\mathbb EZ\rangle$.
For any coupling of \(G\) and \(Z\), by Cauchy--Schwarz,
\[
\mathbb E\langle G,Z-\mathbb EZ\rangle
=
\mathbb E\langle \Sigma^{1/4}G,\Sigma^{-1/4}(Z-\mathbb EZ)\rangle
\le
\sqrt{\operatorname{tr}(\Sigma^{1/2})}
\sqrt{\operatorname{tr}(\Sigma^{1/2})}
=
\operatorname{tr}(\Sigma^{1/2}),
\]
where \(\Sigma^{-1/4}\) is understood as the Moore--Penrose inverse on the range of \(\Sigma\), and the proof is completed by taking the supremum over couplings and extending to the sub-Gaussian case using Corollary~\ref{cor_comp}.

\begin{remark}
Although Corollary~\ref{cor1} resembles a simple Cauchy–Schwarz inequality, it generally fails if the sub-Gaussian assumption is replaced by a second-moment bound.
More precisely, ``the natural metric on $T$ for a centered sub-Gaussian process $(X_t)_{t\in T}$ is dominated by the Euclidean metric''
cannot be replaced by 
``$(X_t)$ is a stochastic process satisfying $\mathbb{E}[|X_t-X_{t'}|^2]\lesssim  \|t-t'\|_2^2$''.
Indeed, consider $T$ as a maximal $\epsilon$-packing of the unit ball in $\mathbb{R}^N$, where $N=3$,
and suppose that $X_t$ is an independent zero mean random variable with $\epsilon^2$ variance for each $t\in T$.
Then it is true that $\mathbb{E}[|X_t-X_{t'}|^2]\lesssim  \|t-t'\|_2^2$.
However, for fixed $N=3$ and vanishing $\epsilon$,
it is possible that $\mathbb{E}[X_Z]=\mathbb{E}[\max_{t\in T}X_t]
\asymp \epsilon \sqrt{|T|}
\asymp \epsilon^{1-N/2}
$ diverges, for certain heavy-tailed $X_t$.
On the other hand, since $Z$ is in the unit ball, $\tr(\sqrt{\Sigma})$ remains bounded from above as $\epsilon$ vanishes, confirming that this is a counterexample.
\end{remark}

\section{Penalization form and new proof of the majorizing measure theorem}\label{sec_new}
In this section,
we prove Corollary~\ref{cor_maj0}.
The next lemma compares an integral with an integral of infima involving a quadratic penalty. We use it to prove Corollary~\ref{cor_maj0} and to bound the risk of empirical risk minimization. 
\begin{lemma}\label{lem_equi}
Suppose that $y\colon(0,\infty)\to[0,\infty]$ is a decreasing function.
Then
\begin{align}
2\int_0^{\infty} y(x)dx
\le 
\int_0^{\infty}\inf_{x>0}\{\alpha^{-2}x^2+y^2(x)\}d\alpha
\le 
4\int_0^{\infty} y(x)dx
\end{align}
\end{lemma}

\begin{proof}
First suppose that $y$ is continuously differentiable, positive, and strictly decreasing, with $\int_0^\infty y(x)\,dx<\infty$. Monotonicity and integrability imply $xy(x)\to0$ at both endpoints, justifying the integration by parts below.
For any $\alpha\in(0,\infty)$, there is a unique $x(\alpha)$ satisfying 
\begin{align}
\alpha=\frac{x(\alpha)}{y(x(\alpha))}.
\label{e_alpha}
\end{align}
Define $\bar{y}(\alpha):=y(x(\alpha))$.
Differentiability of $y(\cdot)$ implies that $x(\cdot)$ and $\bar{y}(\cdot)$ are differentiable, and differentiating \eqref{e_alpha} gives
$1=\frac{\bar{y}(\alpha)x'(\alpha)-x(\alpha)\bar{y}'(\alpha)}{\bar{y}^2(\alpha)}$.
Therefore
\begin{align}
\int_0^{\infty}
(\alpha^{-2}x^2(\alpha)+y^2(x(\alpha)))d\alpha
&=\int_0^{\infty}
2\bar{y}^2(\alpha)d\alpha
\\
&=
\int_0^{\infty}
2(\bar{y}(\alpha)x'(\alpha)-x(\alpha)\bar{y}'(\alpha))d\alpha
\\
&=4\int_0^{\infty}
y(x)dx.
\end{align}
For the lower bound, if $x\le x(\alpha)$ then $y(x)\ge\bar y(\alpha)$, while if $x\ge x(\alpha)$ then $\alpha^{-2}x^2\ge\bar y^2(\alpha)$. Hence
\[
\bar y^2(\alpha)\le\inf_{x>0}\{\alpha^{-2}x^2+y^2(x)\}
\le2\bar y^2(\alpha).
\]
Integrating and using the preceding identity proves both bounds in this case.

For general $y$, including jumps, zero values, and infinite integrals, put
\[
H(\alpha):=\inf_{x>0}\max\{\alpha^{-2}x^2,y(x)^2\},
\qquad a(t):=\sup\{x>0:y(x)>t\},
\]
with $\sup\emptyset=0$. For $t>0$, monotonicity shows that $\{\alpha:H(\alpha)>t^2\}$ has length $a(t)/t$; the possible endpoint does not affect its length. Tonelli's theorem gives
\[
\int_0^\infty H(\alpha)\,d\alpha
=\int_0^\infty 2t\frac{a(t)}t\,dt
=2\int_0^\infty y(x)\,dx.
\]
Now $H(\alpha)\le\inf_{x>0}\{\alpha^{-2}x^2+y(x)^2\}\le2H(\alpha)$ proves the stated bounds in the extended sense.
\end{proof}
\begin{remark}
A similar idea (though in the form of series summation rather than integral) was previously used in proving equivalent forms of the majorizing measure theorem \citep{talagrand1994constructions}.
Lemma~\ref{lem_equi} gives a functional that is convex in one distribution and concave in the other (Remark~\ref{rem_mm}), and simplifies the rate-distortion calculations for ERM.
\end{remark}

By the lattice-supremum convention and finite-net approximation of the majorizing-measure functional, it suffices to prove Corollary~\ref{cor_maj0} for finite $T$. We identify $T$ with points in $\mathbb R^{|T|}$ equipped with the $\ell_2$ metric.
For $i:=i_{\mu}$ in Definition~\ref{def_rd},
Lemma~\ref{lem_equi} implies
\begin{align}
4\sup_{\mu}\int_0^{\infty}\sqrt{i(\epsilon)}d\epsilon
\ge 
\sup_{P_Z}\int_0^{\infty}\min_{P_{Z\hat{Z}}\in\Pi(P_Z,P_Z)}\{
 \alpha^{-2}\mathbb{E}[\|Z-\hat{Z}\|_2^2]
+I(Z;\hat{Z})
\}
d\alpha
\label{e148}
\end{align}
where both suprema are over distributions on $T$
and the minimum is over couplings of $P_Z$ and $P_Z$.
The right side of \eqref{e148} is bounded below by 
\begin{align}
\sup_{P_Z}\inf_{\mu}\int_0^{\infty}\min_{P_{U|Z}}\{ \alpha^{-2}\mathbb{E}[\|Z-U\|_2^2]
+D(P_{U|Z}\|\mu|P_Z)
\}
d\alpha,
\label{e120}
\end{align}
since equality would be achieved if $\inf_{\mu}$ were restricted to $\mu=P_Z$
and $\min_{P_{U|Z}}$ were restricted to $P_{U|Z}$ satisfying $P_{U|Z}P_Z\in \Pi(P_Z,P_Z)$.
Now define, for each $z,\alpha$,
$F_{\mu}(z,\alpha):=\min_{\nu}\{\alpha^{-2}\mathbb{E}_{U\sim\nu}[\|z-U\|_2^2]+D(\nu\|\mu)\}$.
Since the relative entropy is jointly convex (i.e., $D(\frac{\nu_1+\nu_2}{2}\|\frac{\mu_1+\mu_2}{2})\le \frac1{2}D(\nu_1\|\mu_1)+\frac1{2}D(\nu_2\|\mu_2)$),
we see that $F_{\mu}(z,\alpha)$ is convex in $\mu$.
The integral in \eqref{e120} equals $\int_0^{\infty}\mathbb{E}[F_{\mu}(Z,\alpha)]d\alpha$,
which is convex in $\mu$ and linear in $P_Z$.
Apply Sion's minimax theorem \cite{sion1958general} first with the $\alpha$-integral restricted to $[\epsilon,\epsilon^{-1}]$. On this interval the integrand is finite and continuous on the probability simplices (equivalently, $F_\mu(z,\alpha)=-\ln\sum_{u\in T}\mu(u)e^{-\|z-u\|^2/\alpha^2}$).
As $\epsilon\downarrow0$, compactness of the simplex and monotonicity of these truncated integrals allow their minima to pass to the limit. The reverse inequality is the usual minimax inequality. Thus \eqref{e120} equals 
\begin{align}
&\quad\inf_{\mu}\sup_{P_Z}\int_0^{\infty}
\mathbb{E}[F_{\mu}(Z,\alpha)]d\alpha
=
\inf_{\mu}\max_{z\in T}\int_0^{\infty}F_{\mu}(z,\alpha)d\alpha.
\label{e121}
\end{align}
Now fix any $\mu$, $z$, $\alpha$, and let $P_U:=\nu$ achieve the minimum in the definition of $F_{\mu}(z,\alpha)$.
Define 
$r=a\sqrt{\mathbb{E}[\|z-U\|^2]}$, 
where $U\sim P_U$ and $a>1$ can be optimized later.

\mbox{If $\mathbb E\|z-U\|^2=0$}, then $U=z$ almost surely; use positive radii and let $r\downarrow0$ in the subsequent infimum. 
We obtain, from the data processing inequality \citep{thomas2006elements},
\begin{align}
D(P_U\|\mu)
\ge 
d(P_U(B(z,r))\| \mu(B(z,r)))
\end{align}
where $d(a\|b):=a\ln\frac{a}{b}+(1-a)\ln\frac{1-a}{1-b}$ denotes the binary relative entropy function and $B(z,r)$ denotes the ball centered at $z$ with radius $r$.
By the Markov inequality we have $P_U(B(z,r))\ge 1-a^{-2}$.
Then we have 
\begin{align}
d(P_U(B(z,r))\| \mu(B(z,r)))
\ge 
(1-a^{-2})\ln\frac{1-a^{-2}}{\mu(B(z,r))}-a^{-2}\ln a^2
\label{e123}
\end{align}
because if $\mu(B(z,r))>1-a^{-2}$ then the right side of \eqref{e123} is negative and the inequality follows from nonnegativity of relative entropy; if $\mu(B(z,r))\le 1-a^{-2}$ then by monotonicity of the binary relative entropy function we have $d(P_U(B(z,r))\| \mu(B(z,r)))\ge d(1-a^{-2}\| \mu(B(z,r))) 
= (1-a^{-2})\ln\frac{1-a^{-2}}{\mu(B(z,r))}+a^{-2}\ln\frac{a^{-2}}{1-\mu(B(z,r))}$
and \eqref{e123} also holds.
Now the second integral in \eqref{e121} is bounded below by 
\begin{align}
\int_0^{\infty}  
\inf_{r>0}
\left\{\frac{r^2}{a^2\alpha^2}
+\left[
(1-a^{-2})\ln\frac{1-a^{-2}}{\mu(B(z,r))}-a^{-2}\ln a^2
\right]_+
\right\}
d\alpha
\end{align}
which, by Lemma~\ref{lem_equi}, is bounded below by twice
\begin{align}
&\int_0^{{\rm diam}(T)}\sqrt{\left[
(1-a^{-2})\ln\frac{1-a^{-2}}{\mu(B(z,r))}-a^{-2}\ln a^2
\right]_+}\frac{dr}{a}
\nonumber\\
&\ge
\int_0^{{\rm diam}(T)}\sqrt{
(1-a^{-2})\ln\frac1{\mu(B(z,r))}}
\frac{dr}{a}
-
\int_0^{{\rm diam}(T)}\sqrt{h(a^{-2})}
\frac{dr}{a}
\\
&=\frac{\sqrt{1-a^{-2}}}{a}I_{\mu}(z)
-\frac{{\rm diam}(T)}{a}\sqrt{h(a^{-2})}
\end{align}
where we used the basic inequality $\sqrt{[x-y]_+}\ge \sqrt{x}-\sqrt{y}$ for $x,y\ge 0$.
Thus \eqref{e2170}  implies
\begin{align}
2K'\mathbb{E}[\sup_{z\in T}X_z]
\ge 
\inf_{\mu}\sup_z\frac{\sqrt{1-a^{-2}}}{a}I_{\mu}(z)
-\frac{{\rm diam}(T)}{a}\sqrt{h(a^{-2})}.
\end{align}
It follows by considering two points in $T$, that $\mathbb{E}[\sup_{z\in T}X_z]\ge {\rm diam}(T)\mathbb{E}[G\vee 0]=\sqrt{\frac1{2\pi}}{\rm diam}(T)$, where $G\sim \mathcal{N}(0,1)$.
We finish the proof of Corollary~\ref{cor_maj0} by
\begin{align}
\left(2K'
+\frac{\sqrt{2\pi h(a^{-2})}}{a}\right)
\mathbb{E}[\sup_{z\in T}X_z]
\ge 
\frac{\sqrt{1-a^{-2}}}{a}\inf_{\mu}\sup_zI_{\mu}(z).
\end{align}

\begin{remark}\label{rem_mm}
The minimax step \eqref{e120}-\eqref{e121} uses Lemma~\ref{lem_equi} and joint convexity of relative entropy. 
Specifically, the integral in \eqref{e120} is convex in $\mu$ and linear in $P_Z$.
This property does not hold for the functional
\[
\int_0^{\infty}
\min_{P_{U|Z}\colon \mathbb{E}[\|Z-U\|_2^2]\le \epsilon^2}
\sqrt{D(P_{U|Z}\|\mu|P_Z)}\,d\epsilon
\]
before applying Lemma~\ref{lem_equi}.
Moreover, had we not used the relative entropy approach,
then the integral in \eqref{e120} might be replaced by the functional
$\int_0^{\infty}\min_{r_Z}\{ \alpha^{-2}\mathbb{E}[r_Z^2]
+\ln\frac1{\mu(B(Z,r_Z))}
\}
d\alpha$
(see the $\theta_i(T)$ functional defined around  \cite[Proposition~4.3]{talagrand1994constructions}),
where $P_{U|Z}$ is replaced by the conditional distribution of $\mu$ on the  ball $B(Z,r_Z)$, $r_Z$ being a function of $Z$.
This functional is not convex in $\mu$, so the same minimax argument does not apply directly. 
A classical approach that has a similar role as \eqref{e120}-\eqref{e121} is the Fernique convexity argument, which establishes equivalences \eqref{e1.3}-\eqref{e1.4} but loses a constant factor \citep{talagrand2014upper}.
\end{remark}

\section{Proof of ERM over the Sobolev ellipsoid}\label{sec_erm}

In this section we apply the rate-distortion integral to the ellipsoid, and in particular prove Theorem~\ref{thm_ellip}.
Here rate-distortion functions and reproduction channels are taken in the ambient Hilbert space $\ell_2$ (or $\mathbb R$ in the scalar case).
This is sufficient for applying Proposition~\ref{prop1}: the construction of a conditionally independent copy in \eqref{e21} converts an ambient reproduction into one supported on the source index set, at most doubling the distortion radius and without increasing mutual information. This multiplies the integral bound by at most two.

\subsection{Bounding the rate-distortion integral by variance}
\label{sec_toy2}

We bound the rate-distortion integral using the covariance matrix.
One possible approach would be to show that, under the constraint $\cov(Z)=\Sigma$, the Gaussian distribution maximizes the rate-distortion function, namely
\begin{align}
r_{P_Z}(\sigma)
\le r_{\mathcal{N}(0,\Sigma)}(\sigma)
\label{e_rg1}
\end{align}
for all $\sigma^2>0$.
While it is possible to prove \eqref{e_rg1} using Lieb's rotation argument for Gaussian optimality
(see \cite{liu2018forward,courtade2021euclidean} for the analysis of a related problem, namely the reverse Brascamp-Lieb inequality),
we will not pursue that route here.
Instead, we use a simpler alternative based on the penalized form of the rate-distortion integral.
By Lemma~\ref{lem_equi}, we have
\begin{align}
\int_0^{\infty}\sqrt{r_{P_Z}(\sigma)}\,d\sigma
\le\frac1{2}
\int_0^{\infty}
\inf_{P_{U|Z}}
\{\alpha^{-2}\mathbb{E}[\|U-Z\|^2]+I(U;Z)\}\,d\alpha.
\label{e43}
\end{align}
This penalized form has two useful properties.
First, $P_{U|Z}$ is optimized over a set independent of $P_Z$, which permits a minimax formulation.
Second, coordinatewise channels give an upper bound by the sum of the one-dimensional quantities, allowing us to extend scalar bounds to dependent vectors.

Define
\begin{align}
f(t):=
\left\{
\begin{array}{cc}
   t/2,  & 0\le t\le 1, \\[0.3em]
   \frac{1+\ln t}{2}, & t>1.
\end{array}
\right.
\label{e_f}
\end{align}

We first consider the one-dimensional case.
Recall the following result from
\cite[Problem~10.8]{thomas2006elements}:
\begin{lemma}\label{lem6}
Let \(Z\) be a real-valued random variable with \(\operatorname{var}(Z)\le \sigma_{\rm m}^2\). 
Then 
\[
r_{P_Z}(\sigma) \;\le\; r_{N(0,\sigma_{\rm m}^2)}(\sigma)
= \frac12 \Bigl[\ln \frac{\sigma_{\rm m}^2}{\sigma^2}\Bigr]_+
\]
for all \(\sigma>0\).
In particular, among all distributions with variance at most \(\sigma_{\rm m}^2\), the Gaussian maximizes the quadratic rate-distortion function.
\end{lemma}

\begin{lemma}\label{lem_thm8}
For any distribution $P_Z$ on $\mathbb{R}$, we have
\begin{align}
\inf_{P_{U|Z}}
\{\alpha^{-2}\mathbb{E}[|U-Z|^2]+I(U;Z)\}
\le
f(2\alpha^{-2}\sigma_{\rm m}^2)
\end{align}
for any given $\alpha>0$, where $\sigma_{\rm m}^2:=\var(Z)$, interpreted as $+\infty$ when $\mathbb E[Z^2]=+\infty$.
\end{lemma}

\begin{proof}
If $\sigma_{\rm m}=0$, take the constant reproduction $U=Z$; if $\sigma_{\rm m}=\infty$, the right side is infinite. For $0<\sigma_{\rm m}<\infty$, Lemma~\ref{lem6} gives
\begin{align}
\inf_{P_{U|Z}}
\{\alpha^{-2}\mathbb{E}[|U-Z|^2]+I(U;Z)\}
&\le
\inf_{D>0}\left\{\alpha^{-2}D+\frac1{2}\Bigl[\ln \frac{\sigma_{\rm m}^2}{D}\Bigr]_+\right\}
\\
&=
\begin{cases}
\frac12+\frac12\ln\dfrac{2\sigma_{\rm m}^2}{\alpha^2},
&\alpha<\sqrt2\sigma_{\rm m},\\
\alpha^{-2}\sigma_{\rm m}^2,&\alpha\ge\sqrt2\sigma_{\rm m},
\end{cases}
\end{align}
which is exactly $f(2\alpha^{-2}\sigma_{\rm m}^2)$.
\end{proof}

Next, we extend Lemma~\ref{lem_thm8} to vectors followed by a Lipschitz transformation.

\begin{lemma}\label{lem_2}
Assume that $W\in\ell_2$ is a random variable with $\mathbb E\|W\|^2<\infty$, and let $\Sigma:=\cov(W)$.
Let $Z=\phi(W)$, where $\phi:\ell_2\to\ell_2$ is a $1$-Lipschitz function, and suppose that $\mathbb{E}[\|Z\|^2]\le \epsilon$.
Then, for any $\alpha\in(0,\infty)$,
\begin{align}
\inf_{P_{U|Z}}
\{\alpha^{-2}\mathbb{E}[\|U-Z\|^2]+I(U;Z)\}
\le
\alpha^{-2}\epsilon
\wedge
\sum_{i=1}^{\infty} f(2\alpha^{-2}\Sigma_{ii}).
\end{align}
\end{lemma}

\begin{proof}
By taking $U=0$, we immediately obtain
\begin{align}
\inf_{P_{U|Z}}
\{\alpha^{-2}\mathbb{E}[\|U-Z\|^2]+I(U;Z)\}
\le
\alpha^{-2}\epsilon.
\end{align}
Thus it remains to show that
\begin{align}
\inf_{P_{U|Z}}
\{\alpha^{-2}\mathbb{E}[\|U-Z\|^2]+I(U;Z)\}
\le
\sum_{i=1}^{\infty} f(2\alpha^{-2}\Sigma_{ii}).
\label{e192}
\end{align}
It suffices to prove \eqref{e192} in the special case where the $1$-Lipschitz map $\phi$ is the identity, since the left side of \eqref{e192} can only decrease after applying a $1$-Lipschitz map.
For the identity map, use scalar channels for the first $k$ coordinates, conditionally independently given $W$, and set the remaining reproduction coordinates equal to their means.
The chain rule gives $I(U;W)\le\sum_{i=1}^k I(U_i;W_i)$. Applying Lemma~\ref{lem_thm8} coordinatewise yields
\[
\inf_{P_{U|W}}\{\alpha^{-2}\mathbb E\|U-W\|^2+I(U;W)\}
\le\sum_{i=1}^k f(2\alpha^{-2}\Sigma_{ii})
+\alpha^{-2}\sum_{i>k}\Sigma_{ii}.
\]
The last term tends to zero because $\sum_i\Sigma_{ii}<\infty$. Letting $k\to\infty$ proves \eqref{e192}, without subtracting possibly infinite relative entropies.
\end{proof}

\subsection{Warm-up: supremum over an ellipsoid}

This section illustrates how to evaluate the rate-distortion integral. It is not needed for the proof of Theorem~\ref{thm_ellip}.
Recall the ellipsoid $\mathcal E$ in \eqref{e31}. By permuting coordinates, we may assume $a_i$ is nonincreasing. We consider $\sum_i a_i^2<\infty$, since otherwise the right side of \eqref{e32} is infinite. Let $X_z:=\left<G,z\right>$, where $G=(G(1),G(2),\dots)$ has i.i.d.\ standard normal coordinates. On $\mathcal E$, the finite-coordinate truncations converge uniformly almost surely, since their tails are bounded by $(\sum_{i>k}a_i^2G(i)^2)^{1/2}\to0$. Thus the process has continuous paths on $\mathcal E$.
The majorizing measure theorem shows that 
(see \cite[Section~2.13]{talagrand2014upper}): 
\begin{align}
\mathbb{E}[\sup_{z\in\mathcal{E}}
X_z]
\asymp 
\sqrt{\sum_{i=0}^{\infty}2^ia_{2^i}^2}
\asymp
\sqrt{\sum_{i=1}^{\infty}a_i^2}.
\label{e32}
\end{align}
A covering-number calculation \cite{talagrand2014upper} shows
that the Dudley integral is of order $
\sum_{i=0}^{\infty}2^{i/2}a_{2^i}
$,
hence may not be sharp.
While the upper bound in  \eqref{e32} also follows from the Cauchy-Schwarz inequality,
the majorizing measure theorem implies that the upper bound continues to hold when $T$ is the image of the ellipsoid $\mathcal{E}$ under a 1-Lipschitz map, by Gaussian comparison.

The rate-distortion integral recovers the sharp upper bound. 
Suppose that $Z'$ is a random variable taking values in $\mathcal{E}$ and $\Sigma:=\cov(Z')$.
We have
\begin{align}
\int_0^{\infty}\sqrt{r_{P_{Z'}}(\sigma)}\,d\sigma
&\lesssim 
 \sum_{i=1}^{\infty} \sqrt{\Sigma_{ii}}
\label{e_schur}
\\
&\le \|(a_i)_{i=1}^{\infty}\|_2
\label{e_cs}
\end{align}
where 
\eqref{e_schur} follows from \eqref{e43} and Lemma~\ref{lem_2}, noting that $\int_0^{\infty} f(2\alpha^{-2}\Sigma_{ii})d\alpha\asymp \sqrt{\Sigma_{ii}}$,
and \eqref{e_cs} follows from the ellipsoid constraint $\sum_{i=1}^{\infty}\frac{\Sigma_{ii}}{a_i^2}\le 1$ and the Cauchy-Schwarz inequality.
Then for any random variable $Z\in\mathcal{E}$, we have  
$\mathbb{E}[X_Z]\lesssim \|(a_i)_{i=1}^{\infty}\|_2$, which follows from Proposition~\ref{prop1}, by taking $T$ and $\mu'$ therein to be $\mathcal{E}$ and $P_{Z'}$ in \eqref{e_cs}, respectively.

Corollary~\ref{cor1} also gives the upper bound in \eqref{e32}.
Indeed, $\tr(\sqrt{\Sigma})
\le \sum_{i=1}^{\infty} \sqrt{\Sigma_{ii}}$ by the Schur concavity of the square root function and the fact that the eigenvalues majorize the diagonal values. 
By approximating general $P_Z$ by distributions with finite supports, applying Corollary~\ref{cor1} on their finite-dimensional spans, and using the argument from \eqref{e_schur} to \eqref{e_cs},
we can obtain the upper bound in \eqref{e32}.

For the Sobolev ellipsoid, the supremum in \eqref{e32} is of order $\sqrt{\frac1{2\beta-1}}$ as $\beta\downarrow 1/2$.
In contrast, the Dudley integral is 
\begin{align}
\sum_{i=0}^{\infty}2^{i/2}a_{2^i}
=
\sum_{i=0}^{\infty}
2^{-\beta i+i/2}, 
\label{e_dud48}
\end{align}
giving a worse upper bound of order $\frac1{2\beta-1}$
as $\beta\downarrow 1/2$; see \cite[Example 5.12]{wainwright2019high} or \cite[Section~2.13]{talagrand2014upper}.

\subsection{Error bound for ERM}
\label{sec73}

We now prove Theorem~\ref{thm_ellip}.
Recall \eqref{e312}, where we defined the centered sub-Gaussian process $(\chi_z)_{z\in T-T}$, and $T$ is now taken to be a $1$-Lipschitz image of the Sobolev ellipsoid.
The proof is based on the following estimate.

\begin{theorem}\label{thm_mw}
Under the assumptions of Section~\ref{sec34}, with $\beta>1/2$, let $Z$ be a random index taking values in $T-T$, arbitrarily dependent on the data. If $\mathbb{E}[\|Z\|_2^2]\le E$ for some $E>0$, then
\begin{align}
\mathbb{E}[\chi_Z]\lesssim C_{\beta}E^{\frac{2\beta-1}{4\beta}}
\end{align}
where $C_\beta=4c_\beta+1$, with $c_\beta$ defined in \eqref{e84}, and
\[
\lim_{\beta\downarrow \frac12}\sqrt{2\beta-1}\,C_{\beta}<\infty.
\]
\end{theorem}

\begin{remark}\label{rem_trad}
In \cite[Lemma~5.1]{van1990estimating}, the ERM error for the Sobolev space of functions was bounded by localizing Dudley’s integral, under the assumption that $\beta\ge 1$ is an integer.
Suppose we extend that approach to the setting of Theorem~\ref{thm_ellip}.
The covering number satisfies
\begin{align}
\ln N(\delta,\mathcal{E})
\lesssim
\delta^{-1/\beta},
\label{e60}
\end{align}
where throughout this remark we assume $\beta$ is bounded above, so that we focus only on the behavior as $\beta\downarrow 1/2$.
For the Sobolev ellipsoid $\mathcal{E}$, \cite[Example~5.12]{wainwright2019high} proves \eqref{e60}, and the proof shows that the best multiplicative constant in \eqref{e60} is of order independent of $\beta$.
For local covering numbers, \cite{wei2020gauss} proved that the same estimate still holds, with universal multiplicative constant.
We can then calculate the Dudley integral up to scale $\sqrt{E}$, obtaining
\[
\int_0^{\sqrt{E}}\delta^{-1/(2\beta)}\,d\delta
\lesssim
\frac1{2\beta-1}E^{\frac12-\frac1{4\beta}}.
\]
In contrast, our proof yields
\[
\mathbb{E}[\chi_Z]\lesssim \frac1{\sqrt{2\beta-1}}E^{\frac12-\frac1{4\beta}},
\]
which has the same exponent but a better multiplicative constant.
Consequently Dudley’s integral gives
\[
\mathbb{E}[\|\hat m-m\|^2]
\lesssim
(2\beta-1)^{-4\beta/(2\beta+1)}n^{-2\beta/(2\beta+1)},
\]
whose coefficient is of order $(2\beta-1)^{-1}$ as $\beta\downarrow1/2$. Theorem~\ref{thm_ellip} yields a coefficient of order at most $(2\beta-1)^{-1/2}$ as $\beta\downarrow1/2$. As explained in Remark~\ref{rem2}, classical localized Gaussian complexity methods also recover the latter order.
The same loss occurs for the supremum over $\mathcal{E}$ in \eqref{e_dud48}, where Dudley's integral gives order $(2\beta-1)^{-1}$ instead of $(2\beta-1)^{-1/2}$.
\end{remark}

For a random variable $Z\in 2\mathcal{E}$, the estimate
\[
r_{P_Z}(\delta)\lesssim \delta^{-1/\beta},
\]
analogous to \eqref{e60}, still gives the coefficient $(2\beta-1)^{-1}$ rather than $(2\beta-1)^{-1/2}$ after integration, with $\beta$ again bounded above.
The maximizing distribution $P_Z$ depends on the scale $\delta$.
Instead, the coordinatewise channel bound lets us use the penalized form \eqref{e43}, integrate for fixed covariance $\cov(Z)$, and only then optimize over the covariance.

We next bound an integral involving the function $f$ defined in \eqref{e_f}.

\begin{lemma}\label{lem_ft}
For all $t>0$,
we have
$
t\int_0^{1/t} f(1/u^2)\,du \le 4f(t)$.
\end{lemma}

\begin{proof}
For $0<u_0\le 1$,
\begin{align}
\int_0^{u_0} f(1/u^2)\,du
&=
\int_0^{u_0}\frac{1+\ln(1/u^2)}{2}\,du
\\
&=
\frac32u_0+u_0\ln\frac1{u_0}
\\
&\le
\frac32u_0\left(1+\ln\frac1{u_0}\right).
\end{align}
For $u_0>1$,
\[
\int_1^{u_0} f(1/u^2)\,du
=
\int_1^{u_0}\frac1{2u^2}\,du
=
\frac12\left(1-\frac1{u_0}\right),
\]
so that
\begin{align}
\int_0^{u_0} f(1/u^2)\,du
=
\frac32+\frac12\left(1-\frac1{u_0}\right)
=
2-\frac1{2u_0}
\le 2.
\end{align}
The result follows by taking $u_0=1/t$.
\end{proof}

\begin{lemma}\label{lem3}
Suppose that
$
\sum_{i=1}^{\infty}\Sigma_{ii}i^{2\beta}\le 4$,
where $\Sigma_{ii}\ge0$ and $\beta>1/2$.
For any $E>0$, we have
\begin{align}
\int_0^{\infty}
\alpha^{-2}E
\wedge
\sum_{i=1}^{\infty}f(2\alpha^{-2}\Sigma_{ii})
\,d\alpha
\le
C_{\beta}E^{\frac{2\beta-1}{4\beta}},
\end{align}
where $C_\beta=4c_\beta+1$, with $c_\beta$ defined in \eqref{e84}, and
$
\lim_{\beta\downarrow 1/2}\sqrt{2\beta-1}\,C_{\beta}<\infty$.
\end{lemma}

\begin{proof}
Set $\alpha_0:=E^{\frac{1+2\beta}{4\beta}}$.
Then
\begin{align}
\int_{\alpha_0}^{\infty}\alpha^{-2}E\,d\alpha
\le E^{\frac{2\beta-1}{4\beta}}.
\label{e64}
\end{align}
Moreover, by Lemma~\ref{lem_ft}, omitting zero-variance coordinates (which contribute zero),
\begin{align}
\int_0^{\alpha_0}\sum_{i=1}^{\infty}f(2\alpha^{-2}\Sigma_{ii})\,d\alpha
&=
\sum_{i=1}^{\infty}
\sqrt{2\Sigma_{ii}}
\int_0^{\alpha_0/\sqrt{2\Sigma_{ii}}}
f(1/u^2)\,du
\\
&\le
\alpha_0\sum_{i=1}^{\infty}
4f\!\left(\frac{\sqrt{2\Sigma_{ii}}}{\alpha_0}\right).
\label{e66}
\end{align}
The choice of $\alpha_0$ makes both \eqref{e64} and \eqref{e66} of order $E^{(2\beta-1)/(4\beta)}$.

Thus it remains to consider
\begin{align}
M_1:=
\sup_{(\Sigma_i):\,\Sigma_i\ge0}
\sum_{i=1}^{\infty}
f\!\left(\frac{\sqrt{2\Sigma_i}}{\alpha_0}\right)
\end{align}
under the constraint
\[
\sum_{i=1}^{\infty}\Sigma_i i^{2\beta}\le 4.
\]
Also define
\begin{align}
M_2:=
\sup_{\Sigma(\cdot)}
\int_0^{\infty}
f\!\left(\frac{\sqrt{2\Sigma(t)}}{\alpha_0}\right)\,dt
\end{align}
under the constraint
\begin{align}
\int_0^{\infty}\Sigma(t)t^{2\beta}\,dt\le 4,
\label{e_sigma}
\end{align}
where $\Sigma(\cdot)$ is nonnegative and measurable on $(0,\infty)$.
We have $M_2\ge M_1$, since given $(\Sigma_i)$ we may define $\Sigma(t)=\Sigma_{\lceil t\rceil}$.

Let $\bar{f}(x):=f\!\left(\frac{\sqrt{2x}}{\alpha_0}\right)$ for $x\ge 0$, which is a concave function.
We obtain a direct upper bound using the concavity of $\bar f$.
For $\lambda>0$, define $\Sigma(t)$, for $t>0$, by
\begin{align}
\bar f'(\Sigma(t))=\lambda t^{2\beta}.
\label{e_l}
\end{align}
The explicit formulas below determine $\lambda$ so that \eqref{e_sigma} is an equality. For any feasible $\widetilde\Sigma$, the supporting-line inequality gives
\[
\bar f(\widetilde\Sigma(t))\le\bar f(\Sigma(t))
+\lambda t^{2\beta}(\widetilde\Sigma(t)-\Sigma(t)).
\]
After integration, the constraint implies that the objective at $\widetilde\Sigma$ is at most that at $\Sigma$. The calculations below verify that the latter is finite. Thus we need not assume that an optimizer exists. The candidate is continuous and decreasing on $(0,\infty)$; it need not be finite at $t=0$.
There exists a unique $t_0>0$ such that
\[
\frac{\sqrt{2\Sigma(t)}}{\alpha_0}\ge 1, \quad t\in(0,t_0],
\qquad
\frac{\sqrt{2\Sigma(t)}}{\alpha_0}\le 1, \quad t\in[t_0,\infty).
\]
The defining equation for this candidate is equivalent to
\begin{align}
\frac1{4\Sigma(t)}&=\lambda t^{2\beta},
\qquad t\in(0,t_0],\\
\frac1{2\alpha_0\sqrt{2\Sigma(t)}}&=\lambda t^{2\beta},
\qquad t\in[t_0,\infty).
\end{align}
Setting $t=t_0$ in both equations and canceling $\Sigma(t_0)$ gives
\[
\lambda=\frac1{2\alpha_0^2 t_0^{2\beta}}.
\]
We choose $t_0$, and hence $\lambda$, so that
\[
\int_0^{\infty}\Sigma(t)t^{2\beta}\,dt=4.
\]
Indeed,
\begin{align}
4
&=
\int_0^{t_0}\frac1{4\lambda}\,dt
+
\int_{t_0}^{\infty}
\frac1{8\alpha_0^2\lambda^2 t^{2\beta}}\,dt
\\
&=
\frac{t_0}{4\lambda}
+
\frac1{8(2\beta-1)\alpha_0^2\lambda^2 t_0^{2\beta-1}}
\\
&=
\frac{\beta}{2\beta-1}\alpha_0^2 t_0^{2\beta+1}.
\end{align}

We now compute the objective at this candidate:
\begin{align}
\int_0^{t_0}
f\!\left(\frac{\sqrt{2\Sigma(t)}}{\alpha_0}\right)\,dt
&=
\int_0^{t_0}
\frac{1+\ln\bigl(\frac{\sqrt{2\Sigma(t)}}{\alpha_0}\bigr)}{2}\,dt
\\
&=
\int_0^{t_0}
\frac{1+\beta\ln(t_0/t)}{2}\,dt
\\
&=
\frac{t_0(1+\beta)}{2},
\end{align}
and
\begin{align}
\int_{t_0}^{\infty}
f\!\left(\frac{\sqrt{2\Sigma(t)}}{\alpha_0}\right)\,dt
&=
\int_{t_0}^{\infty}
\frac{\sqrt{2\Sigma(t)}}{2\alpha_0}\,dt
\\
&=
\int_{t_0}^{\infty}
\frac{t_0^{2\beta}}{2t^{2\beta}}\,dt
\\
&=
\frac{t_0}{2(2\beta-1)}.
\end{align}

Thus
\[
M_1\le M_2
\le
\left(
\frac{1+\beta}{2}
+\frac1{2(2\beta-1)}
\right)t_0
\le
c_{\beta}\alpha_0^{-2/(1+2\beta)},
\]
where
\begin{align}
c_{\beta}
=
\left(
\frac{1+\beta}{2}
+\frac1{2(2\beta-1)}
\right)
\left(\frac{\beta}{4(2\beta-1)}\right)^{-1/(1+2\beta)}.
\label{e84}
\end{align}
Therefore,
\begin{align}
\int_0^{\infty}
\alpha^{-2}E
\wedge
\sum_{i=1}^{\infty}f(2\alpha^{-2}\Sigma_{ii})
\,d\alpha
\le
4\alpha_0M_1
+
E^{\frac{2\beta-1}{4\beta}}
\le
(4c_{\beta}+1)E^{\frac{2\beta-1}{4\beta}}.
\end{align}
The proof is completed by setting $C_{\beta}=4c_{\beta}+1$.
\end{proof}

\begin{proof}[Proof of Theorem~\ref{thm_mw}]
We apply Proposition~\ref{prop1} using the metric $2\|\cdot\|_2$, which bounds the natural metric of $\chi$; this rescaling only changes a universal constant.
Let $\mu=P_Z$ and let $\mu'$ be supported on $T-T$ with $W_\infty(\mu,\mu')\le\epsilon$, where this Wasserstein distance uses the Euclidean metric. Using a coupling within distance $\epsilon+\eta$ and then letting $\eta\downarrow0$, we have
\[
\mathbb E_{\mu'}\|Z'\|_2^2\le E+\epsilon^2+2\sqrt E\epsilon.
\]
The ellipsoid is compact. Since $(A,B)\mapsto\phi(A)-\phi(B)$ is continuous on $\mathcal E\times\mathcal E$, the measurable selection theorem gives $A,B\in\mathcal E$ as Borel measurable functions of $Z'$ with $Z'=\phi(A)-\phi(B)$.
Set $W=2^{1/2-\beta}(A_1,B_1,A_2,B_2,\ldots)$. Then
\[
\sum_{j=1}^\infty j^{2\beta}\var(W_j)
\le2\sum_{i=1}^\infty i^{2\beta}(\mathbb E A_i^2+\mathbb E B_i^2)\le4.
\]
The map from $W$ to $Z'$ is $2^\beta$-Lipschitz on $\ell_2$, so $2^{-\beta}Z'$ is a $1$-Lipschitz image of $W$.
Apply Lemmas~\ref{lem_2} and~\ref{lem3} to this rescaled image. Changing variables $\alpha=2^\beta a$ gives
\begin{align*}
&\int_0^\infty\inf_{P_{U|Z'}}
\{\alpha^{-2}\mathbb E\|U-Z'\|^2+I(U;Z')\}\,d\alpha\\
&\qquad\le\sqrt2 C_\beta
(E+\epsilon^2+2\sqrt E\epsilon)^{(2\beta-1)/(4\beta)},
\end{align*}
because $2^\beta(2^{-2\beta})^{(2\beta-1)/(4\beta)}=\sqrt2$.
The process $\chi$ has paths continuous in the Euclidean metric since the sample mean belongs to $\ell_2$ almost surely.
Proposition~\ref{prop1}, the reproduction comparison at the start of this section, and \eqref{e43} therefore give
\[
\mathbb E[\chi_Z]\lesssim C_\beta\lim_{\epsilon\downarrow0}
(E+\epsilon^2+2\sqrt E\epsilon)^{(2\beta-1)/(4\beta)}
=C_\beta E^{(2\beta-1)/(4\beta)}.
\]
\end{proof}

\begin{proof}[Proof of Theorem~\ref{thm_ellip}]
Let \(A := \mathbb E\|\hat m - m\|_2^2\).
Taking expectation in \eqref{e312} and applying Theorem~\ref{thm_mw}, we obtain
$
A \lesssim n^{-1/2} C_\beta A^{\frac{2\beta-1}{4\beta}}$.
Rearranging yields
$
A \lesssim C_\beta^{\frac{4\beta}{2\beta+1}}\, n^{-\frac{2\beta}{2\beta+1}}$.
For $A=0$ the claim is immediate. Otherwise this defines $c(\beta)$ up to a universal factor. To check the asserted bound on $c(\beta)$ as $\beta\downarrow1/2$, put $\delta=2\beta-1$. Formula \eqref{e84} gives $C_\beta=O(\delta^{-1/2})$, and
\[
\delta^{1/2}\delta^{-2\beta/(2\beta+1)}
=\delta^{-\delta/(2(\delta+2))}\longrightarrow1.
\]
Thus the stated bound on $c(\beta)$ as $\beta\downarrow1/2$ follows.
\end{proof}

\begin{appendix}

\section{Method of types estimates}
\label{sec_23}
The lifting argument relies on the \emph{method of types}, a standard method in information theory \citep{CsiszarKorner1981} and large deviation analysis \citep{dembo2009large}.
The \emph{type} of a sequence is defined as its empirical distribution. 
Since not all distributions can be a type for a given length of the sequence, we need the following: 
\begin{definition}\label{def2}
Suppose that $\mu$ is a distribution on a finite set $\mathcal{Z}$. We say $\mu$ is \emph{rational} if there exists integer $N>0$ such that $N\mu$ is \emph{integer}, i.e., $N\mu(z)\in\mathbb{Z}$ for any $z\in \mathcal{Z}$. 
\end{definition}
If $N\nu$ is integer, the \emph{type class of $\nu$} is defined as the set of sequences with type $\nu$.
The following result is standard and can be found in \cite[Lemma~2.6]{CsiszarKorner1981}:

\begin{lemma}\label{lem_ck}
Let $\mathcal{S}_N$ be the set of all possible types of sequences in  $\mathcal{Z}^N$, 
where $\mathcal{Z}$ is a finite set with cardinality $n:=|\mathcal{Z}|$.
Then 
$|\mathcal{S}_N|={n+N-1 \choose n-1}\le (N+1)^n$.
Furthermore, if $\mu$ and $\nu$ are distributions on $\mathcal{Z}$, $N\nu$ is integer, and $\mathcal{C}$ denotes the type class of $\nu$,
then
\begin{align}
(N+1)^{-n}
\exp(-ND(\nu\|\mu))
\le 
\mu^{\otimes N}(\mathcal{C})
\le \exp(-ND(\nu\|\mu)).
\end{align}
\end{lemma}
Lemma~\ref{lem_ck} implies the following result, which can be found in
\cite[Chapter~2, Problem~3]{CsiszarKorner1981}:

\begin{lemma}\label{lem_t2}
Suppose that $\mathcal{X}$ and $\mathcal{Y}$ are finite sets, 
$P_{XY}$ is a distribution on $\mathcal{X}\times \mathcal{Y}$, 
and $NP_{XY}$ is integer.
Let $x^N$ be a sequence of type $P_X$,
and let $\mathcal{C}\subseteq \mathcal{Y}^N$ be the type class of $P_Y$.
Define
\begin{align}
\mathcal{C}(x^N):=\{y^N\colon (x^N,y^N) \textrm{ has type $P_{XY}$}\}.
\end{align}
Then we have
\begin{align}
(N+1)^{-|\mathcal{X}||\mathcal{Y}|}
\exp(-NI(X;Y))
\le 
\frac{|\mathcal{C}(x^N)|}{|\mathcal{C}|}
\le 
(N+1)^{|\mathcal{Y}|}\exp(-NI(X;Y)).
\end{align}
\end{lemma}
Note that sequences in $\mathcal{C}(x^N)$ and $\mathcal{C}$ have the same probability under $P_Y^{\otimes N}$.
Taking $\mu=\nu=P_Y$ in Lemma~\ref{lem_ck} yields
\begin{align}
\label{e_ctype}
\quad (N+1)^{-2|\mathcal{X}||\mathcal{Y}|}
\exp(-NI(X;Y))
\le 
P_Y^{\otimes N}(\mathcal{C}(x^N))
\le (N+1)^{|\mathcal{Y}|}\exp(-NI(X;Y)).
\end{align}

\section{From matching to coupling}\label{sec_mc}

\subsection{Proof of Lemma~\ref{lem_von}}

The $\ge$ inequality follows from the definition.
For the $\le$ part, fix an arbitrary $\bar{s}^N\in \mathbb{R}^{nN}$ satisfying $\widehat{P}_{\bar{s}^N}=\mu$.
Denote the supports of $\widehat{P}_{y^N}$ and $\mu$ by $\mathcal{A}$
and $\mathcal{B}$, respectively.
Consider the following function 
\begin{align}
F(M)&=\frac1{N}\sum_{1\le i,j\le N}\|y_i-\bar{s}_j\|_2^2M_{i,j}
\\
&=\frac1{N}\sum_{y\in\mathcal{A},\,s\in\mathcal{B}}
\|y-s\|_2^2
\sum_{i\colon y_i=y}
\sum_{j\colon \bar{s}_j=s}
M_{i,j}
\label{e_63}
\end{align}
on $\mathcal{S}$, which is defined as the set of doubly stochastic matrices $M\in\mathbb{R}^{N\times N}$.
Equation \eqref{e_63} gives $\inf_{M\in \mathcal{S}}F(M)\le W_2^2(\widehat{P}_{y^N},\mu)$.
Since $F$ is linear, by the Birkhoff-von Neumann theorem, there exists a permutation matrix $M$ achieving $\inf_{M\in \mathcal{S}}F(M)$.
This permutation sends $\bar{s}^N$ to an $s^N$, achieving $\inf_{M\in \mathcal{S}}F(M)=\frac1{N}\|y^N-s^N\|_2^2$.
This proves the $\le$ part.

\subsection{Proof of Lemma~\ref{lem:cycle-rounding}}
The proof follows immediately from the following observation, by taking $A:=[NP_{XY}(x,y)]_{x\in\mathcal{X},y\in\mathcal{Y}}$ and $B:=[NQ_{XY}(x,y)]_{x\in\mathcal{X},y\in\mathcal{Y}}$.
\begin{lemma}\label{lem:cycle-rounding0}
Let $A\in\mathbb{R}_{\ge 0}^{m\times n}$ have integer row sums and integer column sums. 
Then there exists $B\in\mathbb{Z}_{\ge 0}^{m\times n}$ with the same row and column sums as $A$ and
\[
\max_{i,j}\,|B_{ij}-A_{ij}|\;\le\;1 .
\]
\end{lemma}

\begin{proof}
Write $F:=A-\lfloor A\rfloor$, so $F_{ij}\in[0,1)$ for all $i,j$. 
Let $r_i:=\sum_j F_{ij}$ and $c_j:=\sum_i F_{ij}$. 
Because the row and column sums of $A$ are integers and $\sum_j\lfloor A_{ij}\rfloor,\sum_i\lfloor A_{ij}\rfloor$ are integers, we have $r_i,c_j\in\mathbb{Z}$.

Form the bipartite graph $G$ with left vertices $\{1,\dots,m\}$ (rows) and right vertices $\{1,\dots,n\}$ (columns),
and put an edge $(i,j)$ iff $F_{ij}\in(0,1)$. 
If some connected component of $G$ containing an edge were a tree, it would have a leaf $v$ incident to a unique edge $e$ of weight $F_e\in(0,1)$, forcing the corresponding row or column sum to be an integer plus $F_e$, contrary to its integrality. 
Hence every component containing an edge has a cycle.

Fix a simple cycle $C$ in $G$. 
Along the edges of $C$, alternately add $+\varepsilon$ and $-\varepsilon$ to the corresponding entries of $F$ (leaving all other entries unchanged), where
\[
\varepsilon:=\min\bigl(\{1-F_{ij}:(i,j)\in C_+\}\cup\{F_{ij}:(i,j)\in C_-\}\bigr),
\]
where $C_+$ and $C_-$ are the alternating sets of edges receiving the positive and negative changes. This preserves all row and column sums, since at each vertex on $C$ the two incident edges change by $+\varepsilon$ and $-\varepsilon$. 
Moreover, all modified entries remain in $[0,1]$, and at least one becomes $0$ or $1$. 
Remove edges for which $F_{ij}$ becomes 0 or 1.
Iterating this operation, which terminates in finitely many steps, we obtain a $0$--$1$ matrix $F^*$ with the same row and column sums as $F$. Then $B=\lfloor A\rfloor+F^*$ has the desired properties.
\end{proof}

\section{Extension of the rate-distortion integral beyond finite $T$}\label{sec_countable}

In this section, 
we extend the proofs in Section~\ref{sec_finite}
beyond the finite case by a limiting argument.
We also prove Proposition~\ref{prop1}.

\subsection{Absolute integrability of $|X_Z|$}
We first prove an auxiliary result about absolute integrability of $|X_Z|$, which will be used later to apply dominated convergence.
\begin{lemma}\label{lem_finite}
Suppose that $(X_t)_{t\in T}$ is a zero-mean sub-Gaussian process on a countable set $T$.
Let $Z\sim \mu$ be a random variable on $T$,
for which $\sigma_{\rm m}:=\sigma_{\rm m}(\mu)<\infty$.
Choose any $t_0\in T$ and write $r:=\sqrt{\mathbb{E}[d^2(Z,t_0)]}<\infty$. Finiteness for every fixed $t_0$ follows from $\sigma_{\rm m}<\infty$ and the triangle inequality; the reference point need not attain the infimum.
\begin{itemize}
\item 
Consider $\bar{T}:=T\times \{-1,1\}$, 
and a process $(\bar{X}_{\bar{t}})_{\bar{t}\in \bar{T}}$, where
$\bar{X}_{\bar{t}}:= s(X_t-X_{t_0})$ for $\bar{t}=(t,s)$.
Let $\bar{d}$ be the natural metric for this sub-Gaussian process on $\bar{T}$,
and $\bar{Z}:=(Z,S)$ be a random variable on $\bar{T}$, satisfying $P_Z=\mu$.
Define 
$\bar{\sigma}_{\rm m}^2
:=\inf_{\bar{t}\in \bar{T}}\mathbb{E}[\bar{d}^2(\bar{Z},\bar{t})]$,
and let $\bar{r}_{P_{\bar{Z}}}(\cdot)$ be the rate-distortion function for $\bar{Z}$ on $(\bar{T},\bar{d})$.
Then 
\begin{align}
\int_0^{\bar{\sigma}_{\rm m}}\sqrt{\bar{r}_{P_{\bar{Z}}}(\sigma)}d\sigma
\le 
\int_0^{\sigma_{\rm m}}\sqrt{r_{\mu}(\sigma)}\,d\sigma
+\sqrt{2\ln2}r.
\label{e49}
\end{align}
\item 
If
$\int_0^{\sigma_{\rm m}} \sqrt{r_{\mu}(\sigma)}\,d\sigma<\infty$, then 
$\mathbb{E}[|X_Z|]<\infty$.
\end{itemize}
\end{lemma}
Lemma~\ref{lem_finite} may be compared with the classical observation that for symmetric processes, $\mathbb{E}[\sup_{t\in T}|X_t|]<\infty$ is equivalent to $\mathbb{E}[\sup_{t\in T}X_t]<\infty$, both equivalent to boundedness of the process in the Gaussian case; see \cite[p2]{talagrand1987regularity}, \cite[Exercise 2.2.2.]{talagrand2014upper} and the references therein.
\begin{proof}
We first prove \eqref{e49}.
For equal signs, the natural metric satisfies
\[
\bar d((t,s),(t',s))\le d(t,t'),\qquad s\in\{-1,1\},\quad t,t'\in T.
\]
Both claims are unchanged if we replace $X_t$ by $X_t-X_{t_0}$, since $X_{t_0}$ is integrable; thus we work with $X_{t_0}\equiv0$.
Then $\bar d(\bar Z,(t_0,1))\le d(Z,t_0)$, so $\bar\sigma_{\rm m}\le r$.
In addition, take any channel $P_{U|Z}$ with $\mathbb{E}[d^2(U,Z)]\le\sigma^2$, and construct $U$ conditionally independently of $S$ given $Z$.
If we set $\bar{U}=(U,S)$, we have $\mathbb{E}[\bar{d}^2(\bar{Z},\bar{U})]\le\mathbb{E}[d^2(Z,U)]\le \sigma^2$, and moreover, the rate-distortion function on $\bar{T}$ satisfies $\bar{r}_{P_{\bar{Z}}}(\sigma)\le I(\bar{U};\bar{Z})=I(U,S;Z,S)
=I(U;Z)+I(S;Z|U)
+H(S|Z)\le I(U;Z)+2\ln2$.
Taking the infimum over channels and using $r_\mu(\sigma)=0$ for $\sigma>\sigma_{\rm m}$, \eqref{e49} follows by the sub-additivity of the square root function.

For the second claim, enumerate $T=\{t_0,t_1,\dots\}$, put $T_k:=\{t_0,\dots,t_k\}$, and set $Z_k=Z$ if $Z\in T_k$ and $Z_k=t_0$ otherwise.
We can show that
\begin{align}
\sqrt{r_{P_{Z_k}}(\sigma)}
\le 
\sqrt{r_{P_Z}(\sigma)+2\ln2}
\le 
\sqrt{r_{P_Z}(\sigma)}+\sqrt{2\ln2}.
\label{e51}
\end{align}
Indeed, for any feasible channel $P_{U|Z}$, setting $E=1\{Z\notin T_k\}$ and  $U_k=U$ if $E=0$, $U_k=t_0$ otherwise, we have
\begin{align}
d(U_k,Z_k)\le d(U,Z),\quad \textrm{a.s.}
\end{align}
and 
\begin{align}
I(U_k; Z_k)&\le I(U_k,E; Z_k)
\\
&=I(E;Z_k)+I(U_k; Z_k|E)
\\
&\le H(E)+I(U_k; Z_k|E)
\\
&\le H(E)+I(U; Z|E)
\\
&\le H(E)+I(UE; Z)
\\
&\le 2H(E)+I(U; Z),
\label{e68}
\end{align}
and taking the infimum establishes \eqref{e51}.
The finite case of Theorem~\ref{thm1} permits the full $T$ as the set of allowed reproduction values: its self-couplings stay on the finite source support, and \eqref{e21} then applies to arbitrary $T$-valued reproductions. Thus we have
\begin{align}
\mathbb{E}[|X_{Z_k}|]
&=
\mathbb{E}[{\rm sign}(X_{Z_k})X_{Z_k}]
\\
&\lesssim \int_0^{\infty}\sqrt{\bar{r}_{P_{\bar{Z}_k}}(\sigma)}\,d\sigma
\\
&\lesssim \int_0^{r}\sqrt{r_{P_{Z_k}}(\sigma)}\,d\sigma
+\sqrt{2\ln2}r,
\label{e61}
\end{align}
where we defined $\bar{Z}_k:=(Z_k,{\rm sign}(X_{Z_k}))$ (taking the sign to be $1$ at zero),
and \eqref{e61} follows from \eqref{e49} by replacing the arbitrary $P_Z$ with $P_{Z_k}$ and noting that $\sigma_{\rm m}(P_{Z_k})\le\sqrt{\mathbb E[d^2(Z_k,t_0)]}\le r$.
Equations \eqref{e51} and \eqref{e61} uniformly bound $\mathbb E|X_{Z_k}|$; since $X_{Z_k}\to X_Z$ almost surely, Fatou's lemma gives $\mathbb E|X_Z|<\infty$.
\end{proof}

\subsection{Extension of Theorem~\ref{thm1} beyond finite $T$}
\label{app_count}
We extend the result of Theorem~\ref{thm1} from the finite case to the countable case.
We can assume that $\int_0^{\sigma_{\rm m}}\sqrt{i_\mu(\sigma)}\,d\sigma<\infty$, since otherwise the first inequality is immediate.
If $\sigma_{\rm m}=0$, all process coordinates in the support of $\mu$ agree almost surely and the result is immediate.
Fix $t_0\in T$, let $r^2=\mathbb E[d^2(Z,t_0)]$, and write $T_k=\{t_0,\ldots,t_k\}$, $p_k=\mu(T_k)$, and $\mu_k=\mu(\cdot\mid T_k)$, omitting initial sets of zero mass.
By \eqref{e21} and Lemma~\ref{lem_finite}, $\mathbb E|X_Z|<\infty$.
Define $Z_k=Z$ if $Z\in T_k$, and otherwise draw $Z_k\sim\mu_k$ independently of the process and $Z$.
Then $P_{Z_k}=\mu_k$.
The removed tail of $X_Z-X_{t_0}$ tends to zero in $L_1$, while the added tail has expectation in absolute value at most
$\sqrt{1-p_k}\sum_t\mu_k(t)d(t,t_0)\le r\sqrt{(1-p_k)/p_k}\to0$ by Cauchy--Schwarz.
Thus
\begin{align}
\lim_{k\to\infty}\mathbb E[X_{Z_k}]=\mathbb E[X_Z].
\label{e62}
\end{align}

To compare self-couplings, fix a symmetric coupling $\pi$ of $\mu$ with itself; symmetrization cannot increase mutual information.
For $A=T_k$, put $p=\mu(A)$ and $b=\pi(A\times A^c)$.
When $b>0$, condition $(V,U)\sim\pi$ on $A\times A^c$, draw $V'$ conditionally independently with the same law as $V$ given $U$, and let $\rho=P_{VV'}$.
Then $\pi_A:=(\pi|_{A\times A}+b\rho)/p$ has both marginals $\mu_k$; omit the second term if $b=0$.
Writing $D_{AA}=\int_{A\times A}d^2\,d\pi$ and $D_{AB}=\int_{A\times A^c}d^2\,d\pi$, we have
\[
\mathbb E_{\pi_A}d^2\le\frac{D_{AA}+4D_{AB}}p\le\frac2p\mathbb E_\pi d^2,
\qquad
I(\pi_A)\le\frac{I(\pi)}p+h(b/p),
\]
where $h$ is binary entropy.
Indeed, the distortion bound follows from the triangle inequality; for information, the mixture indicator contributes at most $h(b/p)$, data processing bounds the mutual information between the conditionally independent copies by that of $\pi(\cdot\mid A\times A^c)$, and the chain rule bounds the weighted conditional mutual informations by $I(\pi)$.
Taking infima gives $i_{\mu_k}(\sigma)\le2i_\mu(\sigma/2)+\ln2$ for $p_k\ge1/2$.
For a fixed coupling of finite distortion, $D_{AB}\to0$, $p_k\to1$, and $b\to0$.
Applying the same construction to arbitrarily near-optimal couplings at a scale $\tau<\sigma$ therefore gives $\limsup_k i_{\mu_k}(\sigma)\le i_\mu(\tau)$.
Letting $\tau\uparrow\sigma$ proves $\limsup_k i_{\mu_k}(\sigma)\le i_\mu(\sigma)$ almost everywhere, since $i_\mu$ is decreasing.

Let $m_k:=\inf_{t\in T_k}\sqrt{\mathbb E_{\mu_k}d^2(Z,t)}$.
Conditional second moments converge for each fixed $t$, so $\limsup_k m_k\le\sigma_{\rm m}$; also $m_k\le\sqrt2r$ eventually.
On this bounded interval, an integrable upper bound is given by
\[
\sqrt{2i_\mu(\sigma/2)+\ln2}.
\]
The dominated $\limsup$ inequality gives
\begin{align}
\limsup_{k\to\infty}\int_0^{m_k}\sqrt{i_{\mu_k}(\sigma)}\,d\sigma
\le\int_0^{\sigma_{\rm m}}\sqrt{i_\mu(\sigma)}\,d\sigma.
\label{e63}
\end{align}
The first inequality of Theorem~\ref{thm1} follows from \eqref{e62}, \eqref{e63}, and the finite case; the second follows from \eqref{e21}.

\subsection{Proof of Proposition~\ref{prop1}}
\label{sec_prop1}
Let $T_0$ be a countable dense subset of $(T,d)$.
For any $\epsilon>0$, let $Z_{\epsilon}\in T_0$ be a random variable such that $|X_Z-X_{Z_{\epsilon}}|\le\epsilon$ and $d(Z,Z_{\epsilon})\le \epsilon$, which is possible by the separability and continuity assumption; measurability follows by choosing the first qualifying index in an enumeration of $T_0$.
Then
$\mathbb{E}[X_Z]\le \mathbb{E}[X_{Z_{\epsilon}}]+\epsilon$,
and $W_{\infty}(\mu,\mu')\le \epsilon$, if $\mu'$ is the distribution of $Z_{\epsilon}$.
As noted in the proof of Lemma~\ref{lem_finite}, reproductions may range over the original $T$.
The infimum defining $\sigma_{\rm m}(\mu')$ is unchanged when the reference points are restricted to the dense set $T_0$: the square root of the second moment of the distance is a $1$-Lipschitz function of the reference point. If the right side of the proposition is finite, Lemma~\ref{lem_finite} gives $\mathbb E|X_{Z_\epsilon}|<\infty$ for some sufficiently small $\epsilon$, and hence $\mathbb E|X_Z|<\infty$.
Therefore the claim follows by the result for the countable case and taking $\epsilon\downarrow0$.

\subsection{Extension of Theorem~\ref{thm2} beyond finite $T$}
\label{app_thm2}

We now use a limiting argument to extend the result to the case of $\mu$ supported on a countable $T$.
Assume that $T=\{t_0,t_1,\dots\}$, with $\mu(t_0)>0$.
First suppose $\sum_t\mu(t)d(t,t_0)<\infty$.
Let $T_k:=\{t_0,t_1,\dots,t_k\}$, $\mu_k=\mu(\cdot\mid T_k)$, and $\tilde\mu_k=\mu(\cdot\mid T_k^c)$; the case $\mu(T_k)=1$ is already finite.
For an arbitrarily near-optimal random index $Z_k$ of law $\mu_k$, construct a Bernoulli random variable $E\sim{\rm Bern}(\mu(T_k))$ independent of the process, let $Z=Z_k$ when $E=1$, and sample $Z\sim\tilde\mu_k$ independently of the process when $E=0$.
The independent tail is integrable by the first-moment assumption, so $Z\sim\mu$ and
\begin{align}
\mathbb E[X_Z]=\mu(T_k)\mathbb E[X_{Z_k}],
\qquad
\mathsf w(\mu)\ge\mu(T_k)\mathsf w(\mu_k).
\label{e122}
\end{align}

For any $\sigma>0$, near-optimal self-couplings for $i_{\mu_k}(\sigma)$ have uniformly tight marginals because $\mu_k\to\mu$ in total variation.
Along a subsequence attaining a finite $\liminf$, they converge to a self-coupling of $\mu$.
Fatou's lemma preserves the distortion bound, and lower semicontinuity of mutual information gives $i_\mu(\sigma)\le\liminf_k i_{\mu_k}(\sigma)$.
On the countable product space, tightness permits a pointwise convergent subsequence, whose limiting masses sum to one. The latter semicontinuity follows by taking finite partitions of each coordinate, using continuity of finite-alphabet mutual information, and then taking the supremum over the partitions.
If the $\liminf$ is infinite, the inequality is immediate.
Fatou's lemma, the finite case, and \eqref{e122} now give
\[
\int_0^\infty\sqrt{i_\mu(\sigma)}\,d\sigma
\le\liminf_k\int_0^\infty\sqrt{i_{\mu_k}(\sigma)}\,d\sigma
\le K'\mathsf w(\mu).
\]
The rate-distortion lower bound follows from \eqref{e21}.

Finally, if $\sum_t\mu(t)d(t,t_0)=\infty$, then $\mathsf w(\mu)=\infty$.
To see this, choose a finite $F\ni t_0$ with $\mu(F^c)\le\mu(t_0)$.
Given the process, assign selection probability $\mu(t)(1+\operatorname{sign}(X_t-X_{t_0}))$ to each $t\notin F$, assign $\mu(t)$ to $t\in F\setminus\{t_0\}$, and assign the remaining probability to $t_0$, taking $\operatorname{sign}(0)=0$.
These probabilities are nonnegative and have expectations $\mu(t)$ by Gaussian symmetry.
Selected tail values are at least $X_{t_0}$, so the negative part is integrable, and cancellation of the $X_{t_0}$ terms gives
\[
\mathbb E[X_Z]=\sqrt{2/\pi}\sum_{t\notin F}\mu(t)d(t,t_0)=\infty.
\]
Thus the extension holds for every countable $\mu$, with the same constant as in the finite case.

\end{appendix}

\begin{acks}[Acknowledgments]
The author thanks Ramon van Handel for feedback on an early version of the manuscript, and for encouraging a new proof of the majorizing measure theorem based on the rate-distortion integral.
The author also thanks Yuting Wei for discussions on the Sobolev ellipsoid example.
The author thanks the anonymous referee for the alternative proof of the risk bound for Lipschitz images of Sobolev ellipsoids outlined in Remark~\ref{rem2}.
\end{acks}

\begin{acks}[Declaration of AI use]
The author used GPT-6 Astra (OpenAI) for proofreading and editorial assistance during the final revision.
\end{acks}

\begin{funding}
This research was supported in part by NSF Grant DMS-2515510.
\end{funding}

\bibliographystyle{spr-ims-number}
\bibliography{references}

@inproceedings{liu2026,
  title={Two-sided bounds for entropic optimal transport via a rate-distortion integral},
  author={Liu, Jingbo},
  booktitle={2026 IEEE International Symposium on Information Theory (ISIT)},
  pages={1--6},
  doi={10.1109/ISIT62367.2026.11653611},
eprint = {2604.14061},
archivePrefix = {arXiv},
primaryClass = {cs.IT},
  year={2026}
}

@article{liuzadik2026bernoulli,
  title={A Bayesian Proof of the {Bernoulli} Theorem},
  author={Liu, Jingbo and Zadik, Ilias},
  journal={arXiv preprint arXiv:2608.11031},
  year={2026}
}

@article{huasongtudose2026convexity,
  title={On {Talagrand}'s Convexity Conjecture},
  author={Hua, Dongming Merrick and Song, Antoine and Tudose, Stefan},
  journal={arXiv preprint arXiv:2605.10908},
  year={2026}
}

@misc{openai2026bernoulli,
  key={Short proof of the Bernoulli conjecture},
  title={Short proof of the {Bernoulli} conjecture},
  year={2026},
  howpublished={\url{https://yangpliu.github.io/repo/bernoulli/paper.pdf}},
  note={Manuscript dated August 11, 2026, generated by an internal model at OpenAI}
}

@article{sion1958general,

  title={On general minimax theorems},

  author={Sion, Maurice},

  journal={Pacific Journal of Mathematics},

  volume={8},

  number={1},

  pages={171--176},

  year={1958}

}

@article{van2025subgaussian,
  title={On the subgaussian comparison theorem},
  author={van Handel, Ramon},
  journal={arXiv preprint arXiv:2512.18588},
  year={2025}
}

@book{chewi2024statistical,
  title={Statistical Optimal Transport: {\'E}cole d'{\'E}t{\'e} de Probabilit{\'e}s de Saint-Flour XLIX--2019},
  author={Chewi, Sinho and Niles-Weed, Jonathan and Rigollet, Philippe},
  series={Lecture Notes in Mathematics},
  volume={2364},
  publisher={Springer},
  address={Cham},
  year={2025},
  doi={10.1007/978-3-031-85160-5}
}

@book{koltchinskii2011oracle,
  title={Oracle inequalities in empirical risk minimization and sparse recovery problems: {\'E}cole d'{\'E}t{\'e} de Probabilit{\'e}s de Saint-Flour XXXVIII-2008},
  author={Koltchinskii, Vladimir},
  series={Lecture Notes in Mathematics},
  volume={2033},
  year={2011},
  publisher={Springer}
}

@misc{maurer2010majorizing,
  title={Majorizing codes and measures},
  author={Maurer, Andreas},
  year={2010},
  howpublished={\url{http://www.andreas-maurer.eu/mmnotes.pdf}},
  note={Preprint}
}

@inproceedings{chu2023majorizing,
  title={Majorizing measures, codes, and information},
  author={Chu, Yifeng and Raginsky, Maxim},
  booktitle={2023 IEEE International Symposium on Information Theory (ISIT)},
  pages={660--665},
  year={2023},
}

@article{wei2020gauss,
  title={{From Gauss to Kolmogorov: Localized measures of complexity for ellipses}},
  author={Wei, Yuting and Fang, Billy and Wainwright, Martin J},
  journal={Electronic Journal of Statistics},
  volume={14},
number={2},
  pages={2988--3031},
  year={2020}
}

@article{talagrand1992simple,
  title={A simple proof of the majorizing measure theorem},
  author={Talagrand, Michel},
  journal={Geometric \& Functional Analysis GAFA},
  volume={2},
  number={1},
  pages={118--125},
  year={1992},
  publisher={Springer}
}

@book{dembo2009large,
  title={Large Deviations Techniques and Applications},
  author={Dembo, Amir and Zeitouni, Ofer},
  year={2009},
  publisher={Springer}
}

@book{wainwright2019high,
  title={High-Dimensional Statistics: A Non-asymptotic Viewpoint},
  author={Wainwright, Martin J},
  series={Cambridge Series in Statistical and Probabilistic Mathematics},
  volume={48},
  year={2019},
  publisher={Cambridge University Press}
}

@article{jiao2018generalizations,
  title={Generalizations of maximal inequalities to arbitrary selection rules},
  author={Jiao, Jiantao and Han, Yanjun and Weissman, Tsachy},
  journal={Statistics \& Probability Letters},
  volume={137},
  pages={19--25},
  year={2018},
  publisher={Elsevier}
}

@incollection{fernique1975regularite,
  author       = {Xavier Fernique},
title = {R\'egularit\'e des trajectoires des fonctions al\'eatoires gaussiennes},
  booktitle = {\'Ecole d'\'Et\'e de Probabilit\'es de Saint-Flour IV -- 1974},
  series       = {Lecture Notes in Mathematics},
  volume       = {480},
  pages        = {1--96},
  year         = {1975},
  publisher    = {Springer},
  address      = {Berlin}
}

@book{ledoux2013probability,
  title={Probability in Banach Spaces: Isoperimetry and Processes},
  author={Ledoux, Michel and Talagrand, Michel},
  year={2013},
  publisher={Springer Science \& Business Media}
}

@incollection{dudley2016vn,
  title={{V. N. Sudakov's work on expected suprema of Gaussian processes}},
  author={Dudley, Richard M},
  booktitle={High Dimensional Probability VII: The Carg{\`e}se Volume},
  series={Progress in Probability},
  volume={71},
  pages={37--43},
  year={2016},
  publisher={Birkh{\"a}user},
  address={Cham},
  doi={10.1007/978-3-319-40519-3_2}
}

@misc{maurey,
  author={Maurey, Bernard},
  title={{Fifty years ago, a theorem by Xavier Fernique}},
  year={2024},
  howpublished={\url{https://webusers.imj-prg.fr/~bernard.maurey/articles/OldiesF.pdf}},
  note={Preliminary draft}
}

@article{raskutti2011minimax,
  title={Minimax rates of estimation for high-dimensional linear regression over $\ell_q$-balls},
  author={Raskutti, Garvesh and Wainwright, Martin J and Yu, Bin},
  journal={IEEE Transactions on Information Theory},
  volume={57},
  number={10},
  pages={6976--6994},
  year={2011},
}

@article{talagrand1994constructions,
  title={Constructions of majorizing measures, {Bernoulli} processes and cotype},
  author={Talagrand, Michel},
  journal={Geometric \& Functional Analysis GAFA},
  volume={4},
  pages={660--717},
  year={1994},
  publisher={Springer}
}

@book{vershynin2018high,
  author    = {Roman Vershynin},
  title     = {High-Dimensional Probability: An Introduction with Applications in Data Science},
  edition   = {2nd},
  year      = {2026},
  publisher = {Cambridge University Press},
  address   = {Cambridge},
  series    = {Cambridge Series in Statistical and Probabilistic Mathematics},
  doi       = {10.1017/9781009490672}
}

@book{thomas2006elements,
  title={Elements of Information Theory},
  author={Thomas M. Cover and Joy A. Thomas},
  year={2006},
  publisher={Wiley-Interscience},
  edition={2nd},
}

@article{aminian2023information,
  title={Information-theoretic characterizations of generalization error for the Gibbs algorithm},
  author={Aminian, Gholamali and Bu, Yuheng and Toni, Laura and Rodrigues, Miguel R. D. and Wornell, Gregory W},
  journal={IEEE Transactions on Information Theory},
  volume={70},
  number={1},
  pages={632--655},
  year={2024},
}

@inproceedings{russo2016controlling,
  title={Controlling bias in adaptive data analysis using information theory},
  author={Russo, Daniel and Zou, James},
  booktitle={Artificial Intelligence and Statistics},
  pages={1232--1240},
  year={2016},
  organization={PMLR}
}

@inproceedings{dziugaite2017computing,
  title={Computing nonvacuous generalization bounds for deep (stochastic) neural networks with many more parameters than training data},
  author={Dziugaite, Gintare Karolina and Roy, Daniel M},
  booktitle={Proceedings of the 33rd Conference on Uncertainty in Artificial Intelligence (UAI)},
  year={2017}
}

@article{van1990estimating,
  title={Estimating a regression function},
  author={van de Geer, Sara},
  journal={The Annals of Statistics},
  pages={907--924},
  year={1990},
volume={18},
number={2},
}

@article{talagrand1987regularity,
  title={Regularity of {Gaussian} processes},
  author={Talagrand, Michel},
  journal={Acta Mathematica},
  volume={159},
  number={1},
  pages={99--149},
  year={1987},
  publisher={Springer}
}

@misc{van2014probability,
  title={Probability in High Dimension},
  author={van Handel, Ramon},
  year={2016},
  howpublished={\url{https://web.math.princeton.edu/~rvan/APC550.pdf}},
  note={APC 550 lecture notes, Princeton University, version of December 21, 2016}
}

@book{talagrand2014upper,
  title={Upper and Lower Bounds for Stochastic Processes: Decomposition Theorems},
  author={Talagrand, Michel},
  volume={60},
  edition={2nd},
  year={2021},
  publisher={Springer},
  doi={10.1007/978-3-030-82595-9}
}

@article{liu2018forward,
  title={A forward-reverse {Brascamp-Lieb} inequality: Entropic duality and {Gaussian} optimality},
  author={Liu, Jingbo and Courtade, Thomas A and Cuff, Paul W and Verd{\'u}, Sergio},
  journal={Entropy},
  volume={20},
  number={6},
  pages={418},
  year={2018},
}

@article{bu2020tightening,
  title={Tightening mutual information-based bounds on generalization error},
  author={Bu, Yuheng and Zou, Shaofeng and Veeravalli, Venugopal V},
  journal={IEEE Journal on Selected Areas in Information Theory},
  volume={1},
  number={1},
  pages={121--130},
  year={2020},
  publisher={IEEE}
}

@article{xu_raginsky,
author={A. Xu and M. Raginsky}, 
title={Information-theoretic analysis of generalization capability of learning
algorithms},
journal={Advances in Neural Information Processing Systems},
volume={30},
year={2017}, 
pages={2524--2533},
}

@book{CsiszarKorner1981,
  author    = {Imre Csisz{\'a}r and J{\'a}nos K{\"o}rner},
  title     = {Information Theory: Coding Theorems for Discrete Memoryless Systems},
  publisher = {Academic Press},
  year      = {1981},
  address   = {New York},
  edition   = {1st}
}

@article{asadi2018chaining,
  title={Chaining mutual information and tightening generalization bounds},
  author={Asadi, Amir and Abbe, Emmanuel and Verd{\'u}, Sergio},
  journal={Advances in Neural Information Processing Systems},
  volume={31},
  year={2018}
}

@article{courtade2021euclidean,
  title={Euclidean forward--reverse {Brascamp--Lieb} inequalities: Finiteness, structure, and extremals},
  author={Courtade, Thomas A and Liu, Jingbo},
  journal={The Journal of Geometric Analysis},
  volume={31},
  number={4},
  pages={3300--3350},
  year={2021},
  publisher={Springer}
}

@article{mendelson2003performance,
  author={Mendelson, Shahar},
  title={On the Performance of Kernel Classes},
  journal={Journal of Machine Learning Research},
  volume={4},
  pages={759--771},
  year={2003}
}

@article{dirksen2015tail,
  author={Dirksen, Sjoerd},
  title={Tail bounds via generic chaining},
  journal={Electronic Journal of Probability},
  volume={20},
  number={53},
  pages={1--29},
  year={2015},
  doi={10.1214/EJP.v20-3760}
}

\end{document}